\documentclass[11pt,letterpaper]{amsart}
\usepackage[utf8]{inputenc}
\usepackage{amssymb,amsmath}
\usepackage[foot]{amsaddr}
\usepackage{hyperref}
\usepackage{enumitem}
\usepackage[normalem]{ulem}
\usepackage{fullpage}

\usepackage{pgfplots}
\usetikzlibrary{math,calc,arrows,shapes.misc,positioning}
\pgfplotsset{compat=1.16}

\numberwithin{equation}{section}

\makeatletter
\@namedef{subjclassname@2020}{\textup{2020} Mathematics Subject Classification}
\makeatother

\newtheorem{thm}{Theorem}[section]
\newtheorem{lemma}[thm]{Lemma}
\newtheorem{prop}[thm]{Proposition}
\newtheorem{cor}[thm]{Corollary}
\newtheorem{conjecture}[thm]{Conjecture}

\theoremstyle{definition}
\newtheorem{defn}[thm]{Definition}

\theoremstyle{remark}
\newtheorem{rmk}[thm]{Remark}

\def\diam{\operatorname{diam}}
\def\diag{\operatorname{diag}}
\def\Leb{{\operatorname{Leb}}}

\def\Orb{\operatorname{Orb}}
\def\musrb{\mu_{\mbox{\tiny \textup{SRB}}}}
\def\wroot{W_{\mbox{\tiny \textup{root}}}}
\newcommand{\srb}{\mbox{\tiny \textup{SRB}}}

\def\bP{{\mathbb{P}}}
\def\bR{{\mathbb{R}}}
\def\bW{{\mathbb{W}}}

\def\eps{{\varepsilon}}
\def\ve{{\varepsilon}}
\def\vf{{\varphi}}

\def\cA{{\mathcal{A}}}

\def\cC{{\mathcal{C}}}
\def\cG{{\mathcal{G}}}
\def\cK{{\mathcal{K}}}

\def\cM{{\mathcal{M}}}
\def\cO{{\mathcal{O}}}
\def\cR{{\mathcal{R}}}
\def\cS{{\mathcal{S}}}

\def\fA{{\mathfrak{A}}}
\def\fG{{\mathfrak{G}}}
\def\fR{{\mathfrak{R}}}
\def\fS{{\mathfrak{S}}}
\def\fP{{\mathfrak{P}}}

\def\brN{{\breve{N}}}

\def\brfG{{\breve{\mathfrak{G}}}}

\def\oW{{\overline{W}}}

\newcommand*\circled[1]{\tikz[baseline=(char.base)]{
    \node[shape=rounded rectangle,draw,inner sep=0.25em] (char) {$\mathrm{#1}$};}}

\let\originalleft\left
\let\originalright\right
\renewcommand{\left}{\mathopen{}\mathclose\bgroup\originalleft}
\renewcommand{\right}{\aftergroup\egroup\originalright}

\title[Rates of mixing for the MME in dispersing billiards]{Rates of  mixing for the measure of maximal entropy of dispersing billiard maps}

\author{Mark F. Demers$^1$}
\address{$^1$Department of Mathematics, Fairfield University, Fairfield CT 06824, USA}
\email{mdemers@fairfield.edu}

\author{Alexey Korepanov$^2$}
\address{$^2$Laboratoire de Probabilit\'es, Statistique et Mod\'elisation (LPSM),
Sorbonne Universit\'e, Universit\'e de Paris, 4 Place Jussieu, 75005 Paris, France}
\email{korepanov@lpsm.paris}


\date{27 October 2023}

\begin{document}

\begin{abstract}
    In a recent work, Baladi and Demers constructed a measure of maximal entropy
    for finite horizon dispersing billiard maps and proved that it is unique, mixing and
    moreover Bernoulli.
    We show that this measure enjoys natural probabilistic properties for H\"older continuous
    observables, such as at least polynomial decay of correlations and
    the Central Limit Theorem.

    The results of Baladi and Demers are subject to a condition of sparse recurrence to singularities.
    We use a similar and slightly stronger condition, and it has a direct effect
    on our rate of decay of correlations.
    For billiard tables with bounded complexity (a property conjectured to be generic),
    we show that the sparse recurrence condition is always satisfied and the correlations decay at a
    super-polynomial rate.
\end{abstract}

\maketitle

\section{Introduction}

Dispersing billiards were introduced by Sinai~\cite{Si63, Si70} as mechanical models of 
particles for which Boltzmann's hypotheses could be verified.  They have become 
central to mathematical physics as prototypical models of hyperbolic systems with
singularities.  Such billiards are defined by placing finitely many convex obstacles 
having boundary with strictly positive curvature on the
two-torus and following the motion of a point particle moving at unit speed and undergoing specular reflections
at collisions.  The collision-to-collision map is called the billiard map associated to the
configuration of obstacles and it is known to preserve a smooth invariant measure.

Over the ensuing decades, a wealth of information has been obtained regarding the 
statistical properties of the system with respect to the smooth invariant measure:
it is ergodic, mixing~\cite{Si70} and indeed Bernoulli~\cite{GaOr}, enjoys exponential
decay of correlations~\cite{Y98}, and a host of limit theorems including the 
Central Limit Theorem~\cite{BSC} and related invariance principles~\cite{MN05}.
See~\cite{CM06} for an excellent exposition of the subject and additional references.

Despite these successes, only recently has attention been given to other invariant measures
for dispersing billiards. Denote by $T$ the billiard map corresponding to a finite horizon
Sinai billiard and let $\cM_n$ be the set comprising the maximal
domains of continuity of $T^n$.  Let $\# A$ denote the cardinality of a set
$A$ and following \cite{BD20} define
\begin{equation}
    \label{eq:hM}
    h = \lim_{n \to \infty} \frac 1n \log \big( \# \cM_n \big)
    .
\end{equation}
We call $h$ the {\em topological entropy} of $T$.  The main result of \cite{BD20} is that
under a condition of sparse recurrence to singularities (see \eqref{eq:sparse}), $T$ satisfies a variational
principle, 
\begin{equation}
    \label{eq:h}
    h = \sup \bigl\{ h_\mu(T) : \mu \mbox{ is a $T$-invariant Borel probability measure} \bigr\}
    ,
\end{equation}
where $h_\mu(T)$ denotes the Kolmogorov-Sinai entropy of $\mu$.
Moreover, there exists a unique measure $\mu_0$, called the measure of maximal entropy (MME),
which attains the supremum, is $T$-adapted, Bernoulli and positive on open sets.  

However, left open in~\cite{BD20} is the rate of mixing of 
$\mu_0$ with respect to a reasonable class of observables.
In the present work we prove a polynomial bound on the rate of decay of correlations for $\mu_0$
and derive an almost-sure invariance principle (which implies e.g., the Central Limit Theorem).
Furthermore, assuming that a bounded complexity conjecture holds for generic billiard tables,
we prove that generically correlations decay at a super-polynomial rate.

In this context, the question naturally arises whether our rates of mixing are optimal or, in fact,
$\mu_0$ enjoys exponential decay of correlations like the smooth invariant measure.
Indeed, for many classes of smooth systems the MME mixes exponentially,
even when hyperbolicity is weak, as long as the topological entropy of the system is positive.
Notable examples are some classes of multimodal maps in~\cite{Bruin Todd, IT10},
and topologically transitive $C^\infty$ surface diffeomorphisms
with positive topological entropy in~\cite{BCS}.
In the piecewise monotone setting, the recent preprint \cite{Tiozzo} proves topological mixing for
unimodal maps whose topological entropy is greater than $\frac 12 \log 2$, while
\cite{Baladi} proves that topological mixing implies exponential mixing for the MME of such maps.
We are not aware of any natural examples where the MME is known to mix slower than exponentially.
 
It is interesting that for dispersing billiards the situation seems to be special due to the
blow up of the derivative near tangential collisions; in particular, exponential mixing of the MME
is known to follow from topological mixing if hyperbolicity dominates complexity
for piecewise hyperbolic maps without such a blow up~\cite{D21}.
One can view the MME as belonging to a family of equilibrium states corresponding to
the geometric potentials $-t \log J^uT$, $t \in \mathbb{R}$, where $J^uT$ is the Jacobian
of $T$ along unstable manifolds.  We define the {\em pressure} corresponding to such potentials by
\[
    P(t) = \sup \Bigl\{ h_\mu(T)  - t \int \log J^uT \, d\mu : \mu \mbox{ is a $T$-invariant probability measure} \Bigr\} .
\]
The recent work~\cite{BD21} identifies a $t_* > 1$
such that $P(t)$ is real analytic on $(0, t_*)$ and for each $t \in (0, t_*)$ there is a unique
equilibrium state $\mu_t$ (i.e., $\mu_t$ achieves the above supremum) and enjoys
exponential decay of correlations on H\"older observables.  
Moreover, under the sparse recurrence condition, $\lim_{t \downarrow 0} P(t) = h$.

Yet if the billiard table has a periodic orbit with a grazing collision, then $P(t) = \infty$ for
$t < 0$ so that the pressure function signals a phase transition at $t=0$.
Since the spectral gap for the associated transfer operators in the construction of 
$\mu_t$ tends to 0 as $t$ decreases to 0 in \cite{BD21}, it is reasonable to expect 
a subexponential rate of decay of correlations at $t=0$.  The present work proves a polynomial
upper bound on the rate of decay with exponent 
expressed as a function of $h/s_0$, where $s_0$ is a parameter controlling
recurrence to singularities~\eqref{eq:sparse}.
However, the question of lower bounds for the decay rate remains open. 
 
The paper is organized as follows.  In Section~\ref{sec:setting} we define our setting and assumptions
precisely, recall some facts about dispersing billiards and state our main results
regarding decay of correlations and an almost-sure invariance principle.  In Section~\ref{sec:cylinders} we describe our main construction which consists of counting
proper returns to a reference {\em magnet rectangle}.  This is similar in spirit to the construction of a 
Young tower, but rather than estimating the measure of points making a proper return to the base at time $n$,
we control the cardinalities of distinct itineraries.
The key estimates are contained in Propositions~\ref{prop:YCb} and \ref{prop:YTb}.
In Section~\ref{sec:symbolic} we introduce a symbolic model which captures counts of
returning itineraries from Section~\ref{sec:cylinders}, and on which probabilistic results
like the Central Limit Theorem are standard.
In Section~\ref{sec:final} we relate the symbolic model to the billiard and complete
proofs of the main results.

\section{Setting and Statement of Results}
\label{sec:setting}

We begin with a precise description of the class of systems we will study.
Let $\cO_i \subset \mathbb{T}^2$ be strictly convex pairwise disjoint closed sets 
(called obstacles or scatterers)
whose boundaries are $C^3$ curves with
strictly positive curvature, $i = 1, \ldots, d$.  
The billiard flow is defined by the motion of a point particle
moving in straight lines at unit speed in the billiard table 
$Q = \mathbb{T}^2 \setminus ( \cup_{i=1}^d \cO_i )$
and undergoing specular reflections (angle of incidence equals angle of reflection) at 
collisions with the boundary $\partial Q$.  If we identify 
ingoing and outgoing collisions, then the flow is continuous on $Q \times \mathbb{S}^1$.

The discrete-time billiard map $T$ is defined as the collision map on the 
global Poincar\'e section $\partial Q$.  We adopt standard coordinates at collisions,
$x = (r, \vf)$ where $r$ denotes position on $\partial Q$ parametrized (clockwise) by arc length and
$\vf$ denotes the angle made by the post-collision velocity vector with the outward normal
to $\partial \cO_i$ at position $r$.
Then $M = \cup_{i=1}^d \mathbb{S}_i \times [-\pi/2, \pi/2]$, where $\mathbb{S}_i$ is the circle
of length $|\partial \cO_i|$, is the phase space for the billiard map $T$.

Let $\tau(x)$ denote the time between collisions at $x$ and $T(x)$ in $Q$.  
We will study tables $Q$ which satisfy the {\em finite horizon} condition: 
there exist no trajectories making only tangential collisions.
This implies in particular that 
there exists $\tau_{\max} < \infty$ such that $\tau(x) \le \tau_{\max}$ for all $x \in M$.  The
fact that the $\cO_i$ are closed, disjoint and convex implies also that there exists $\tau_{\min}>0$ such that
$\tau \ge \tau_{\min}$. We will refer to this class of billiards as 
{\em finite horizon dispersing billiards} throughout.  We do not consider tables with corner points.


\subsection{Singularities, hyperbolicity and sparse recurrence}
\label{ssec:sparse}

Denote by $\cS_0 = \{(r, \varphi) \in M : \varphi = \pm \pi/2\}$ the set of tangential collisions
and  let $\cS_{\pm n} = \cup_{k=0}^n T^{\mp k} \cS_0$.
Then $\cS_{\pm n}$ is the discontinuity set of $T^{\pm n}$.
It is a standard fact (see, for example, \cite[Sect.~4.9]{CM06})
that $\cS_n$ comprises a finite collection of $C^2$ curves with compact closures in $M$.
Moreover, these curves obey the property of {\em continuation of singularities}: 
For each $n > 0$, every curve $S \subset \cS_{\pm n} \setminus \cS_0$ is part of a
monotonic and piecewise smooth curve $\tilde S \subset  \cS_{\pm n} \setminus \cS_0$
that terminates on $\cS_0$.

With this notation, the collection $\cM_n$ from~\eqref{eq:hM} is the partition of $M \setminus \cS_n$
into its maximal connected components.  Note that $A \in \cM_n$ if and only if $T^n(A) \in \cM_{-n}$, 
where $\cM_{-n}$ is the corresponding partition of $M \setminus \cS_{-n}$.
The quantity $h$ measures the exponential rate of growth in the cardinality of the domains
of continuity of $T^n$.
Indeed, 
Baladi and Demers~\cite{BD20} proved that the limit in \eqref{eq:h} exists and coincides with topological entropy
defined in other natural ways, for example by counting Bowen's separated
and spanning sets. 

The existence of a unique measure of maximal entropy is proved in \cite{BD20} under the
following condition, which can be seen as a type of sparse recurrence to singularities
relative to $h$.  For $0 < \vf_0 < \pi/2$ and $n_0 \ge 1$, let $s_0(\vf_0, n_0) \le 1$ be
the smallest number such that any orbit segment for $T$ of length $n_0$ makes at most $s_0 n_0$
collisions with $|\vf| > \vf_0$.  It is a consequence of the finite horizon condition that we may
always choose
$\vf_0$ and $n_0$ so that $s_0 < 1$.  Indeed, in a table without triple tangencies
(a generic condition) one has $s_0 \le 2/3$.
We assume:  
\begin{equation}
\label{eq:sparse}
\mbox{There exist $\vf_0 < \pi/2$ and $n_0 \ge 1$ such that $h > s_0 \log 2$.}
\end{equation}
This assumption seems to be quite mild:  we are not aware of any dispersing billiard table
for which it fails. See \cite[Section~2.4]{BD20} for a discussion of this condition and explicit verifications
for several popular models.
The factor $\log 2$ appears due to the fact that at nearly tangential collisions, one has
$|TW| \sim |W|^{1/2}$ for a local unstable manifold $W$, where $| \cdot |$ denotes 
Euclidean length of the curve.
In Section~\ref{ssec:results}, we state a conjecture due to~\cite{Balint Toth}
on the growth of complexity for finite horizon dispersing billiard tables.
We prove that under this conjecture, $s_0$ can be chosen arbitrarily small for typical tables,
thus~\eqref{eq:sparse} is typically satisfied. This allows us to strengthen the rate of decay of
correlations to super-polynomial in Corollary~\ref{cor:super}.

Despite the presence of singularities, $T$ enjoys uniform hyperbolicity in the following sense.
There exist stable and unstable cones, $\cC^s(x)$ and $\cC^u(x)$, which are strictly
contracted by the dynamics, $DT(x) \cC^u(x) \subsetneq \cC^u(Tx)$ for all $x \notin \cS_1$ and
$DT^{-1}(x) \cC^s(x) \subsetneq \cC^s(T^{-1}x)$ for all $x \notin \cS_{-1}$.
Indeed, the cones have a particularly simple (and global) definition,
\begin{equation}
    \label{eq:cU}
    \cC^u(x) = \bigl\{ (dr, d\vf) \in \mathbb{R}^2 : \cK_{\min} \le \tfrac{d\vf}{dr} \le \cK_{\max} + \tau_{\min}^{-1} \bigr\} ,
\end{equation}
where $\cK_{\min}>0$ and $\cK_{\max}<\infty$ denote the minimum and maximum curvatures of the
$\partial \cO_i$ in $Q$.  A similar formula holds for $\cC^s$ \cite[Section~4.4]{CM06}.

We call a $C^2$ curve $W \subset M$ an {\em unstable curve} if its tangent vector at each point lies in $\cC^u$.  
Stable curves are defined analogously.
The singularities $\cS_{\pm n}$ align with the stable and unstable cones in the
following sense:  For each $n >0$, each smooth component $S \subset \cS_n \setminus \cS_0$ is an unstable
curve, while each smooth component $S \subset \cS_{-n} \setminus \cS_0$ is a stable
curve~\cite[Sect.~4.9]{CM06}.
Thus $\cS_n$ is uniformly transverse to $\cC^u$ and $\cS_{-n}$ is uniformly transverse to $\cC^s$.

Define $\Lambda = 1 + 2\cK_{\min} \tau_{\min}$.  Then $\Lambda$ is the hyperbolicity
constant which governs the minimum rate of contraction and expansion in the stable and
unstable cones.  There exists $C_e > 0$ such that for all $n \ge 0$ (\cite[eq. (4.19)]{CM06}),
\begin{equation}
    \label{eq:hyp}
    \begin{aligned}
        \| DT^n(x) v \| & \ge C_e \Lambda^n \| v \| \quad \forall v \in \cC^u(x)
        , \\
        \| DT^{-n}(x) v \| & \ge C_e \Lambda^n \| v \| \quad \forall v \in \cC^s(x)
        . 
    \end{aligned}
\end{equation}


\subsection{Statement of main results}
\label{ssec:results}

Throughout this section we assume that $T$ is a finite horizon dispersing billiard map as
described above and that the sparse recurrence condition \eqref{eq:sparse} holds.

For a function $u \colon M \to \mathbb{R}$ and $\gamma \in (0,1]$, define
\begin{equation}
    \label{eq:holder}
    | u |_{C^\gamma} = \sup_{x \in M} |u(x)| + \sup_{x \neq y \in M} \frac{|u(x) - u(y)|}{d(x,y)^\gamma}
    ,
\end{equation}
where $d( \cdot, \cdot)$ denotes the Riemannian distance in $M$.  We call
$u$ a H\"older observable with exponent $\gamma$ if $| u |_{C^\gamma} < \infty$. 
Let $C^\gamma(M)$ denote the set of H\"older observables with exponent $\gamma$.

Our main results are as follows.

\begin{thm}
    \label{thm:decay}
    Let $s_0 \in (0,1)$ be chosen according to \eqref{eq:sparse}.
    Suppose that $h > s_0 \log 4$ and $\ve \in \bigl( 0, \frac{h}{s_0 \log 2} - 2 \bigr)$. 
    Then for every $\gamma \in (0,1]$, all $u,v \in C^\gamma(M)$ and $n \geq 0$,
    \[
        \biggl| \int u \, v \circ T^n \, d\mu_0 - \int u \, d\mu_0 \int v \, d\mu_0 \biggr|
        \leq C_{\gamma,\eps} |u|_{C^\gamma} |v|_{C^\gamma} \, n^{- \frac{h}{ s_0 \log 2} + 2 + \eps}
        ,
    \]
    where $C_{\gamma,\eps}$ is a constant depending on $\gamma$, $\eps$
    and on the billiard $T \colon M \to M$.
\end{thm}

An important quantity used to control the effect of cutting due to the singularity sets $\cS_n$ 
is the {\em complexity}. 
For $n \in \mathbb{Z}$, let $K_n$ denote the maximum number of curves in $\cS_n$ that intersect at one point.  
For $n>0$, since $\cS_n$ is uniformly transverse to the unstable cone, if $W$ is a sufficiently short unstable curve,
then there can be at most $K_n +1$ connected components of $W \setminus \cS_n$.  Thus the rate of
growth of $K_n$ is a measure of how unstable curves are fragmented due to the presence of singularities.
This, in turn, informs the global rate of growth of $\#\cM_n$ (see the verification of (F2) in Section~\ref{sub:verify} below or
\cite[Proposition~5.5]{BD20}).

It is a classical
result due to Bunimovich (see, for example \cite[Lemma~5.2]{C01}) that for finite horizon dispersing billiards,
\begin{equation}
\label{eq:linear}
\mbox{There exists $K>0$ such that $K_n \le K|n|$ for all $n \in \mathbb{Z}$.}
\end{equation}
The fact that the complexity bound is subexponential is essential to the uniform bounds on $\#\cM_n$ obtained in \cite{BD20} and
is essential to the present work.  Yet, it is conjectured that for generic
dispersing billiard tables, the complexity bound is even stronger.
Indeed, for a given billiard table, we say the {\em sequence of complexities is bounded} if there exists $K>0$ such that
$K_n \le K$ for all $n \in \mathbb{Z}$.

\begin{conjecture}{\em (\cite[Conjecture~3.3]{Balint Toth})}
    \label{conj}
    The sequence of complexities $K_n$ is bounded for typical
    finite horizon dispersing billiard configurations.
\end{conjecture}

Typical here can be taken in either an algebraic\footnote{Consider a family of scatterers on 
$\mathbb{T}^2$ whose boundaries are given by algebraic equations in some parameters; 
then the set of non-typical tables is conjectured to lie in a lower-dimensional subspace.}
or topological\footnote{The set of tables in the $C^3$ topology that has bounded complexity
is conjectured to be a countable intersection of open and dense sets.} sense.
See~\cite[Section~3.3]{Balint Toth} for a more extended discussion of this conjecture, 
in 2 and higher dimensions.  Notice that the conjecture does not hold if the table has 
a periodic orbit with a grazing collision; however~\eqref{eq:sparse} can still hold in this case,
and does in all known examples (see~\cite[Section~2.4]{BD20}).
We prove the following corollary in Proposition~\ref{prop:s_0}.

\begin{cor}
\label{cor:super}
Assume Conjecture~\ref{conj} holds in some sense of typicality.  
Then for a typical dispersing billiard table, \eqref{eq:sparse} holds with $s_0$ arbitrarily small and
the rate of decay of correlations for $\mu_0$ is super-polynomial.
\end{cor}

\begin{defn}
    \label{def:ASIP}
    We say that a real-valued random process $S_n$ satisfies
    the \emph{almost sure invariance principle} (ASIP) with rate $o(n^p)$ if
    without changing its law, $\{S_n\}$ can be redefined on a probability space which
    supports a Brownian motion $W_t$ such that
    \[
        S_n = W_n + o(n^p)
        \qquad \text{almost surely.}
    \]
\end{defn}

\begin{rmk}
    Useful corollaries of the ASIP with rate $o(n^p)$ with $p < 1/2$
    include the Central Limit Theorem, the Law of the Iterated Logarithm
    and their functional versions, see for example Philipp and Stout~\cite{PS75}.
\end{rmk}

\begin{thm}
    \label{thm:ASIP}
    Let $v \colon M \to \bR$ be H\"older with $\int v \, d\mu_0 = 0$ and let
    $S_n = \sum_{k=0}^{n-1} v \circ T^{k}$. Consider $S_n$ as a random process on
    the probability space $(M, \mu_0)$.
    If $h > s_0 \log 8$, then for each $p > \bigl(\frac{h}{s_0 \log 2} - 1 \bigr)^{-1}$
    the process $S_n$ satisfies the ASIP with rate $o(n^p)$.
\end{thm}


\subsection{Discussion of sparse recurrence conditions in Theorems~\ref{thm:decay} and \ref{thm:ASIP}}

As mentioned above, condition \eqref{eq:sparse} is not known to fail for any 
dispersing billiard table.  In this section, we briefly discuss the slightly stronger conditions
assumed in Theorems~\ref{thm:decay} and \ref{thm:ASIP}. 

Assuming Conjecture~\ref{conj} and using Corollary~\ref{cor:super},
for typical billiard tables, $s_0$ can be chosen
arbitrarily small so that the conditions $h > s_0 \log 4$ and $h > s_0 \log 8$ are generically satisfied.
Yet there are tables with grazing periodic orbits for which $s_0$ has a positive minimum value.
For such tables, numerical studies indicate that these
assumptions are indeed satisfied in some popular models of dispersing billiards.  We mention 
two such studies, both of which are described in \cite[Section~2.4]{BD20}.

While both of these studies approximate the entropy of $\musrb$ by approximating
the positive Lyapunov exponent $\chi_{\srb}^+$ of the system, we will use this as a proxy for verifying our conditions since according to \cite[Theorem~2.4]{BD20}
\[
    h > h_{\musrb} = \chi_{\srb}^+
    .
\]  
Thus, for example, verifying $\chi_{\srb}^+ > s_0 \log 8$ will imply $h > s_0 \log 8$.

\begin{figure}[ht]
\begin{tikzpicture}[x=8mm,y=8mm]

\filldraw[fill=black!20!white, draw=black]  (5.6,0) arc (0:90:1.6);
\filldraw[fill=black!20!white, draw=black!20!white]  (5.6,0) -- (4,1.6) -- (4,0) -- cycle;
\filldraw[fill=black!20!white, draw=black]  (8,1.6) arc (90:180:1.6);
\filldraw[fill=black!20!white, draw=black!20!white]  (8,1.6) -- (6.4,0) -- (8,0) -- cycle;

\filldraw[fill=black!20!white, draw=black]  (6.4,4) arc (180:270:1.6);
\filldraw[fill=black!20!white, draw=black!20!white]  (6.4,4) -- (8,2.4) -- (8,4) -- cycle;

\filldraw[fill=black!20!white, draw=black]  (4,2.4) arc (270:360:1.6);
\filldraw[fill=black!20!white, draw=black!20!white]  (4,2.4) -- (5.6,4) -- (4,4) -- cycle;

\draw[dashed] (4,0) rectangle (8,4) ;

\filldraw[fill=black!20!white, draw=black]  (6,2) circle (.9);
\draw (6,2) -- (6.9,2);
\node at (6.4,1.8){\small $r$};

\draw (4,0) -- (5.13, 1.13);
\node at (4.4, 1) {\small$R$};

\node at (6,-.7){\small$(a)$};

 \filldraw[fill=black!20!white, draw=black] (13.7,.7) circle (.6);
\filldraw[fill=black!20!white, draw=black] (15.1,.7) circle (.6);
\filldraw[fill=black!20!white, draw=black] (16.5,.7) circle (.6);
\filldraw[fill=black!20!white, draw=black] (13,1.9) circle (.6);
\filldraw[fill=black!20!white, draw=black] (14.4,1.9) circle (.6);
 \filldraw[fill=black!20!white, draw=black] (15.8,1.9) circle (.6);
  \filldraw[fill=black!20!white, draw=black] (17.2,1.9) circle (.6); 
\filldraw[fill=black!20!white, draw=black] (13.7,3.1) circle (.6);
\filldraw[fill=black!20!white, draw=black] (15.1,3.1) circle (.6);
\filldraw[fill=black!20!white, draw=black] (16.5,3.1) circle (.6);

\draw[dashed] (13.7,3.1) -- (14.4,1.9) -- (15.8,1.9) -- (15.1,3.1) -- cycle;

\draw[decoration={brace,raise=2.5pt},decorate] 
  (13.7,3.1) -- node[above=4pt] {$d$} (15.1,3.1);

\node at (15.1,-.7){\small$(b)$};

 \end{tikzpicture}
\caption{(a) Sinai billiard on a square lattice with scatterers of radii $r < R$ from \cite{garrido}. 
(b) Sinai billiard on a triangular lattice with angle $\pi/3$ from \cite{baras}, scatterers of radius 1 and distance $d$ between centers of adjacent scatterers. 
The boundary of a single cell is indicated by dashed lines in both tables.}
\label{fig:tables}
\end{figure}
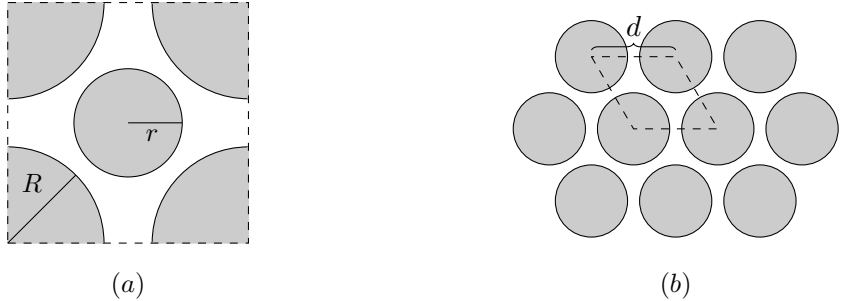

In \cite{garrido}, Garrido studied a family of dispersing billiards formed by two circular
scatterers of radii $r < R$ on a square lattice with side length 1 (see Figure~\ref{fig:tables}(a)).
As observed in \cite{BD20}, for such tables satisfying the finite horizon condition,
every double tangency is followed by at least two nontangential collisions, so we may
choose $\vf_0$ and $n_0$ so that $s_0 \le 1/2$.
Setting $R = 0.4$, Garrido
computed values of $\chi_{\srb}^+$ for $r$ ranging from $r = 0.1$ (when $\tau_{\min}=0$)
to $r = \frac{\sqrt{2}}{2} - 0.4 \approx 0.3071$ (when $\tau_{\max}=\infty$).  All recorded values of
$\chi_{\srb}^+$ are greater than $0.82 > \frac{1}{2} \log 4$ so that the hypothesis of
Theorem~\ref{thm:decay} are numerically satisfied for all such tables.  The values of 
$\chi_{\srb}^+$ remain above $1.039 \approx \frac 12 \log 8$ from $r = 0.1$ until 
about $r = 0.28$, which is very close to the infinite horizon threshold.  So the hypothesis
of Theorem~\ref{thm:ASIP} is numerically satisfied for all but a small segment of such tables.

We remark that choosing $R = 0.4$ and $r = \frac 12 - \frac{0.4}{\sqrt{2}}$ produces a
table with a period 8 orbit making 4 grazing collisions with the central scatterer of radius $r$
and 4 collisions with angle $\vf = \pi/4$ with the scatterer of radius $R$.  For such a
table, $s_0 = 1/2$ and $\chi_{\srb}^+ \approx 1.3 > \frac 12 \log 8$. 

In \cite{baras}, Baras and Gaspard studied a class of dispersing billiards formed by a single 
circular scatterer of radius 1 in a triangular lattice with distance $d$ between the centers of 
adjacent scatterers (Figure~\ref{fig:tables}(b)).  The distance varied between $d=2$ (when $\tau_{\min}=0$)
and $d= 4/\sqrt{3}$ (when $\tau_{\max} = \infty$).  As with the previous class of tables, one can
choose $s_0 \le 1/2$.  All tables tested by Baras and Gaspard above $d = 2.01$
satisfied $\chi_{\srb}^+ > \frac 12 \log 4$.  The table corresponding to $d = 2.01$ fails this condition.
For the condition $\chi_{\srb}^+ > \frac 12 \log 8$, the tables tested with $d > 2.15$ satisfy
this condition.

In summary, the majority of tables tested numerically in \cite{baras, garrido} satisfy the
strongest condition $\chi_{\srb}^+ > s_0 \log 8$ and nearly all of them satisfy
$\chi_{\srb}^+ > s_0 \log 4$.  We conclude by remarking that even when this condition fails
numerically, it does not imply that the hypotheses of Theorems~\ref{thm:decay} and
\ref{thm:ASIP} fail since in fact, $h > \chi_{\srb}^+$.
To estimate $h$ directly is an interesting project in its own right and outside the
scope of this paper.


\section{Billiard coding}
\label{sec:cylinders}

In this section we build a collection of \emph{cylinders} (or symbolic itineraries)
for the map $T$ which we later use to construct a semiconjugacy with a symbolic model.
We do this by building a type of return scheme or Young tower based on iterates of
a reference local unstable manifold which can be thickened into a rectangle.
The main results are Propositions~\ref{prop:YCb} and~\ref{prop:YTb}.

From \cite[Theorem~4.66]{CM06} it follows that Lebesgue almost every point in $M$ has a
stable and unstable manifold of positive length (this also holds for 
the measure of maximal entropy $\mu_0$ \cite[Corollary~7.4]{BD20}).
Moreover, \cite[Proposition~4.29 and Corollary~4.61]{CM06} prove that there exists $C_u >0$ such that every
local unstable manifold is $C^{1+\text{Lip}}$ with Lipschitz constant bounded by $C_u$.


\begin{defn}
    We call connected pieces of stable/unstable manifolds of positive length
    \emph{stable/unstable leaves.}\footnote{Global stable and unstable manifolds may be disconnected for billiards.
        Then unstable leaf is simply shorthand for local unstable manifold.}
    A leaf may or may not include its endpoints.  We use the term maximal stable/unstable leaf to refer to a leaf
    whose two endpoints belong to the set $\cup_{n \ge 0} \cS_n$ or $\cup_{n \le 0} \cS_n$, respectively.
\end{defn}

\begin{defn}
    A stable leaf $W$ is \emph{preperiodic} if there is a finite set of
    stable leaves $\{W_k\}$ such that $T^n(W) \subset \cup_k W_k$ for each $n \geq 0$.
    The minimal cardinality of $\{W_k\}$ is the \emph{preperiod} of $W$.
\end{defn}

For an unstable leaf $W$, its image $T(W)$ is a union of unstable leaves.
For $n \ge 0$, let $\cG_n(W)$ denote the collection of maximal connected components of $T^n(W)$,
excluding isolated points.\footnote{See \cite[Figure~1]{BD20} for an example of how isolated points can arise
from multiple tangencies.}
We denote the Euclidean length of $W$ by $|W|$, and by $|T^n(W)|$ the sum of lengths of the leaves.

As in~\cite[Section~7.11]{CM06}, we define:
\begin{defn}
    A \emph{solid rectangle} (or just rectangle) $D \subset M$ is a closed domain bounded by two stable
    and two unstable leaves. We call the two stable/unstable leaves the {\em stable/unstable boundary} of $D$
    and denote it by $\partial_s D / \partial_u D$.
    We denote by $\mathring{D}$ the interior of $D$.
\end{defn}

\begin{defn}
    We say that an unstable leaf $W$ {\em fully crosses} a (solid) rectangle
    $D$ if $\oW \cap \partial_s D$ consists of two points, where $\oW$ is the closure of $W$,
    and a stable leaf $W$ {\em fully crosses} $D$ if $\oW \cap \partial_u D$ consists of two points.
\end{defn}

\begin{defn}
    We say that a solid rectangle $R$ \emph{u-crosses} a rectangle $D$ if the unstable leaves
    of $\partial_u R$ fully cross $D$.
    If $R$ u-crosses $D$ and $R \subset D$, we say that $R$ is a \emph{u-subrectangle} of $D$.
    Similarly we define \emph{s-crossings} and \emph{s-subrectangles}.
\end{defn}

Assume that $s_0<1$ from~\eqref{eq:sparse} has been chosen
(and so $\vf_0$ and $n_0$ are fixed as well).
Recall that $\Lambda = 1 + 2\cK_{\min}\tau_{\min}$ from \eqref{eq:hyp}
denotes the minimum expansion factor along
unstable curves $W$, and that $\cG_n(W)$ denotes the maximal connected components of $T^n(W)$, excluding isolated
points.

Our construction of the symbolic model relies on the following facts.\footnote{We formulate (F1)-(F4)
in terms of unstable leaves because this is what we shall use in our construction, but equivalent statements hold for unstable curves, and under
time reversal, for stable curves as well.}

\begin{enumerate}[label=(F\arabic*), ref=F\arabic*]
    \item\label{A:hprime}
        For every $h' > 0$ and $N' \ge 1$ with $\frac{\log(K N'+1)}{N'} \le h'$,
        there exists $\delta' > 0$ such that if $W$ is an unstable leaf, then
        \[
            |W| \leq \delta'
            \quad \text{implies} \quad
            \# \cG_{N'}(W) \le e^{h' N'}
            .
        \]
        Further we assume that such $h'$, $\delta'$ and $N'$ are fixed, and that $h' < h - s_0 \log 4$.
        We are interested in small $h'$ because it plays the role of $\eps$ in Theorem~\ref{thm:decay}.
    \item\label{A:delta_1}
        There are $C_1 > 0$ and $\delta_1 \in (0, \delta']$ such that
        if $W$ is an unstable leaf, then:
        \begin{enumerate}[label=(\alph*), ref=\theenumi.\alph*]
            \item\label{A:upper-exp} $\# \cG_n(W) \leq C_1 e^{h n}$ for all $n \geq 0$.
            \item\label{A:super-growth} If $W \cap \cS_n = \emptyset$ (i.e.\ $T^n$ is continuous on $W$)
                for some $n \geq 0$, then $|T^n(W)| \leq C_1 |W|^{2^{-s_0 n}}$.
            \item\label{A:lower-exp} For every $\delta > 0$ there are $C(\delta) > 0$ and $N(\delta) \geq 1$ so that
                if $|W| \geq \delta$, then
                \(
                    \# \{U \in \cG_n(W) : |U| \geq 3 \delta_1 \}
                    \geq C(\delta) e^{n h}
                \)
                for all $n \geq N(\delta)$.
        \end{enumerate}
    \item\label{A:transverse}
        Unstable manifolds are transverse to singularities, so that each unstable manifold
        intersects the boundary of each $\eps$-neighborhood of $\cS_1$ in at most $N_\cS$ points,
        and there is $C_2 > 0$ such that the maximum length of an unstable leaf contained in
        an $\eps$-neighborhood of $\cS_1$ is bounded by $C_2 \eps$.
    \item\label{A:D}
        There are $\delta_2 > 0$, $N_1 \geq 1$ such that for every sufficiently small $b_0 > 0$
        there are $N_2 \geq 1$, $\delta_3 > 0$ and a solid rectangle $D \subset M$ such that:
        \begin{enumerate}[label=(\alph*), ref=\theenumi.\alph*]
            \item\label{A:D-B}
                $D$ does not intersect the $b_0$-neighborhood of $\cS_1$.
            \item\label{A:D-period}
                The stable boundary of $D$ comprises two preperiodic stable leaves, with preperiod bounded above by $N_1$,
                which never returns to the interior of $D$,
                    i.e.\ $T^n(\partial_s D) \cap \mathring{D} = \emptyset$ for all $n \geq 0$.
            \item\label{A:D-long} $D$ is long in the unstable direction: every unstable leaf fully crossing
                $D$ has length at least $\delta_2$.
            \item\label{A:D-narrow} $D$ is narrow in the stable direction:
                every stable curve contained in $D$ has length at most $C_e b_0 / 2$;
                the constant $C_e$ is from~\eqref{eq:hyp}.
                Moreover, if $x \in D$ satisfies $d(T^n(x), \cS_1) \geq b_0 \Lambda^{-n}$ for all $n \geq 0$,
                        then $x$ belongs to a stable manifold which fully crosses $D$.
                        (See Remark~\ref{rmk:D-narrow}.)
                        Each point on the stable boundaries of $D$ satisfies this condition.
            \item\label{A:D-crossing} If $W$ is an unstable leaf with $|W| \geq \delta_1 / 3$,
                where $\delta_1>0$ is from \eqref{A:delta_1},
                then for each $n \geq N_2$ there is a subleaf $W' \subset W$ such that:
                \begin{itemize}
                    \item $T^k(W')$ does not intersect the $b_0$-neighborhood of $\cS_1$
                        for all $0 \leq k \leq n$. In particular, $T^n$ is smooth on $W'$.
                    \item $T^{n}(W')$ fully crosses $D$.
                    \item $d(T^{n}(W'), \partial_u D) \geq \delta_3$.
                \end{itemize}
            \item\label{A:D-rD} $r_D = \min \{ n \geq 1 : T^n(x) \in D \text{ for some } x \in D\}$,
                the minimal first return time to $D$, is large so that
                $\Lambda^{-r_D} \leq 1/2$, and $|T^n(\gamma)| \leq \delta_3 / 2$ for every stable leaf $\gamma \subset D$ and $n \geq r_D$.
        \end{enumerate}
        Further we assume that $\delta_2$, $\delta_3$, $b_0$, $N_2$ and $D$
        are chosen and fixed, and that $b_0$ is \emph{sufficiently} small.
        The precise requirements on $b_0$ are contained in the proofs. 
\end{enumerate}    

\begin{defn}
    \label{def:b0}        
    We denote by $B_n$ the $b_0 \Lambda^{-n}$-neighborhood of $\cS_1$.
\end{defn}

\begin{rmk}
    \label{rmk:D-narrow}
    In~(\ref{A:D-narrow}), the condition $d(T^n(x), \cS_1) \geq b_0 \Lambda^{-n}$ for all $n \geq 0$
    is a way to guarantee that $x$ belongs to a stable manifold
    with endpoints at a distance at least $C b_0$ from $x$, where $C$
    depends only on the billiard (see Lemma~\ref{lem:min-stable}).
    Hence to satisfy~(\ref{A:D-narrow}) it is enough to ensure
    that $D$ is sufficiently narrow in the stable direction.
\end{rmk}

\subsection{Verification of (\ref{A:hprime})--(\ref{A:D})}
\label{sub:verify}

In order to count elements of $\cG_n(W)$, we rely on the results\footnote{In fact, we use the time reversal of the estimates from \cite{BD20},
    since these are phrased in terms of the evolution of stable curves under $T^{-n}$ while we will be concerned with the
    evolution of unstable leaves under $T^n$.} of~\cite{BD20}.
To this end, we translate between our $\cG_n(W)$ and
the collection $\cG_n^{\delta_0}(W)$ used in~\cite{BD20}.
Here $\delta_0$ is a positive constant, chosen in the same way as our $\delta'$.
For an unstable
curve $W$, the set $\cG_1^{\delta_0}(W)$ comprises the maximal connected components of $T(W)$,
with curves longer than length $\delta_0$ subdivided to have length between $\delta_0$ and $\delta_0/2$.  
$\cG_n^{\delta_0}(W)$ is then defined for $n>1$ by applying this construction inductively to each
element of $\cG_{n-1}^{\delta_0}(W)$.

\begin{lemma}
    \label{lem:equal}
    There exists $C > 0$ such that for any $\delta_0>0$, unstable leaf $W$ and $n \ge 0$,
    \[
        C \delta_0 \# \cG_n^{\delta_0}(W)
        \le \# \cG_n(W)
        \le \# \cG_n^{\delta_0}(W)
        .
    \]
\end{lemma}

\begin{proof}
    The upper bound is trivial since $\cG_n^{\delta_0}(W)$ requires extra cuts not present in
    $\cG_n(W)$.  The lower bound follows from hyperbolicity~\eqref{eq:hyp}:
    $|T^n(W)| \geq C_e \Lambda^n |W|$.  Thus the
    extra pieces in $\cG_n^{\delta_0}(W)$ which come from an artificial subdivision at time $k\le n$,
    if not cut again by a genuine singularity, have length at least $C_e \delta_0/2$.
    This implies that a single curve in $\cG_n(W)$ can contain at most
    $2M_u /(C_e \delta_0)$ elements of $\cG_n^{\delta_0}(W)$, where $M_u$ is
    the maximum length of an unstable curve in $M$.
\end{proof}

With the equivalence established, we start by verifying (\ref{A:hprime}).
Let $K>0$ be from the linear complexity bound~\eqref{eq:linear}.
Given $h'>0$ and $N' \geq 1$ such that $K N' + 1 \le e^{N' h'}$, we choose
$\delta' > 0$ such that any unstable leaf $W$ with $|W| \le \delta'$
intersects $\cS_{N'}$ in at most $K N'$ points
and thus $T^{N'}(W)$ has at most $K N' + 1$ connected components.
Then
\[
    |W| \leq \delta'
    \qquad \text{implies} \qquad
    \# \cG_{N'}(W) \leq e^{N' h'}
    ,
\]
as required.

Next,
\begin{itemize}
    \item[(\ref{A:upper-exp})] follows from the exact exponential growth
        of $\# \cM_n$ given by~\cite[Proposition~4.6]{BD20}:
        \begin{equation}
            \label{eq:exact}
            e^{h n} \le \# \cM_n = \# \cM_{-n} \le C_1 e^{h n}
            ,
        \end{equation} 
        and the fact that each element of $\cG_n(W)$ is contained in
        a unique element of $\cM_{-n}$.
    \item[(\ref{A:lower-exp})] follows from the combinatorial Growth Lemma \cite[Lemma~5.2]{BD20}
        and the uniform lower bound on the growth of $\# \cG_n(W)$ given by \cite[Proposition~5.5]{BD20}.
    \item[(\ref{A:super-growth})] is equation\footnote{We thank J.~Carrand for pointing out
        that \cite[eq.~(5.3)]{BD20} needs additional justification in order to apply $s_0$
        from~\eqref{eq:sparse} to the evolution of unstable (or stable) curves.
        This is accomplished in the proof of~\cite[Lemma~3.1]{carrand}.}
        (5.3) in the proof of~\cite[Lemma~5.1]{BD20}.
\end{itemize}

The transversality requirements in~(\ref{A:transverse}) can be found in~\cite[Section~4.5]{CM06}.

Construction of the rectangle $D$ for~(\ref{A:D}) occupies the rest of this subsection.
First we justify Remark~\ref{rmk:D-narrow} (similar statements like~\cite[Lemma~4.67]{CM06}
or \cite[Section~8, Sublemma~1]{Y98} are not quite convenient for our goals 
{since the time reversal of the first would replace $\cS_1$ by $\cS_{-1}$, while the second
has less explicit control on the length of the local stable manifold).

\begin{lemma}
    \label{lem:min-stable}
    If $x \in M$ satisfies $d(T^n(x), \cS_1) \geq b \Lambda^{-n}$ for some $b > 0$ and all $n \geq 0$,
    then $x$ belongs to a local stable manifold with endpoints on a distance at least $C_e b$ from
    $x$. Here $C_e >0$ is the expansion constant from~\eqref{eq:hyp}.
\end{lemma}

\begin{proof}
    We slightly modify the proof of~\cite[Lemma~4.67]{CM06}.
    We may assume that $x \in M \setminus \cup_{n \geq 0} \cS_n$. Fix $n \geq 1$.
    Let $Q$ be the connected component of $M \setminus \cS_n$ containing $x$.
    Let $W_n \subset T^n(Q)$ be an arbitrary stable curve passing through $T^n(x)$ and terminating
    at the opposite sides of $T^n(Q)$, and let $W = T^{-n}(W_n)$.
    The map $T^{-n} \colon W_n \to W$ is continuous,
    and $W$ is a stable curve passing through $x$.
    Both $W$ and $W_n$ do not include their endpoints, just as $Q$ does not contain its boundary.

    If $y$ is an endpoint of $W$, then there is $0 \leq m < n$ such that
    $T^m(y) \in \cS_1$ and $T^m$ is continuous on $W \cup \{y\}$.
    By uniform hyperbolicity~\eqref{eq:hyp},
    \[
        d(x,y)
        \geq C_e \Lambda^m d(T^m(x), T^m(y))
        \geq C_e \Lambda^m d(T^m(x), \cS_1) \ge C_e b
        .
    \]
    Since this holds for all $n \geq 1$
    and the local stable manifold through $x$ is the unique stable curve belonging to the connected component
    of $M \setminus \cS_n$ containing $x$ for all $n \ge 1$, the lemma follows.
\end{proof}

Next, we recall some basic properties of rectangles and Cantor rectangles.

\begin{defn}
    \label{def:cantor}
    For a solid rectangle $D$, we denote by $D_*$ the maximal Cantor rectangle contained in $D$:
    \[
        D_* = D \cap \bigl( \cup_{W^s \in \fS^s} \cup_{W^u \in \fS^u} W^s \cap W^u \bigr)
        ,
    \]
    where $\fS^s(D)$ and $\fS^u(D)$ are the families of all stable/unstable leaves
    which fully cross $D$.
    We say that a rectangle $D$ is \emph{thick} if $\Leb(D_*) > 0$.
\end{defn}

\begin{rmk}
    By Sinai's Fundamental Theorem~\cite[Theorem~5.70]{CM06}, thick rectangles exist everywhere
    in $M$: each open subset of $M$ contains a thick rectangle.
\end{rmk}

\begin{rmk}
    From the proof of~\cite[Proposition~7.11]{BD20}, if $\Leb (D_*) > 0$ then
    $\mu_0(\fS^s(D))>0$, where $\mu_0$ is the measure of maximal entropy.
\end{rmk}

From~\cite[Section~7.12]{CM06} we extract the following two lemmas.

\begin{lemma}
    \label{lem:CM-W}
    Suppose that $\fR$ is a thick (solid) rectangle. Then for every $\delta > 0$ there is $n_\delta > 0$
    so that every unstable leaf $W$ with $|W| > \delta$ has a subleaf $W'$ on which
    $T^{n_\delta}$ is continuous and $T^{n_\delta}(W')$ fully crosses $\fR$.
\end{lemma}

\begin{proof}
    The proof follows that of~\cite[Proposition~7.83]{CM06}, with $\fR_*$ in place
    of the \emph{magnet rectangle.} Although the magnet rectangle is built from
    a specific family of stable leaves, that is not used
    for the proof of the referenced proposition; it is sufficient that $\fR$ is thick.
    
    In particular, if $\fR$ is thick, then by Lebesgue density we can find a u-subrectangle
    $R$ of $\fR$ such that $\rho^s(R) > 0.99$, where 
    \[
        \rho^s(R) := \inf_{x \in \fR_* \cap R} \frac{|W^s(x) \cap \fR_* \cap R |}{|W^s(x) \cap R |}
        ,
    \]
    denotes the minimal density of intersections with elements of $\fS^u(\fR)$ in each stable leaf of $\fS^s(\fR)$
    restricted to $R$; here $W^s(x)$ is the maximal stable leaf containing $x$.
    Then \cite[Lemma~7.90]{CM06} applies to the rectangle $R$ and if $T^{n_\delta}(W)$
    fully crosses $R$, then $T^{n_\delta}(W)$ fully crosses $\fR$.
\end{proof}

\begin{lemma}
    \label{lem:CM-R}
    Suppose that $\fR$ is a thick rectangle and $R \subset \fR$ is a rectangle with
    $\Leb(R \cap \fR_*) > 0$. Then for some $n_{R} \geq 1$ and each $n \geq n_{R}$
    there exists a thick s-subrectangle $\fR'$ of $\fR$ which s-crosses $R$
    such that $T^n$ is continuous on $\fR'$ and $T^n(\fR')$ u-crosses both $R$ and $\fR$.
    Further, there is a stable leaf $W$ fully crossing both $R$ and $\fR$ such that
    $T^n(W) \subset W$.
\end{lemma}

\begin{proof}
    We look again into the proof of~\cite[Proposition~7.83]{CM06}. 
    From \cite[Corollary~7.89]{CM06}, for any $\eps > 0$, there exists $\delta_* >0$ 
    and a subset $\fP_* \subset R \cap \fR_*$ with $\Leb(\fP_*)>0$ such that for every stable manifold $W$
    with $|W| \leq \delta_*$ and $W \cap \fP_* \neq \emptyset$ 
    we have $|W \cap R \cap \fR_*| / |W| \geq 1 - \eps$.

    Next, $\Leb(R \cap \fR_*) > 0$ and the mixing property of $T$
    guarantee that $\Leb(\fP_* \cap T^n(R \cap \fR_*) ) > 0$ for all sufficiently large $n$.
    Let $\fS^s_* \subset \fS^s(\fR)$ denote the set of stable leaves properly crossing $\fR$
    which contain at least one point in $T^{-n}(\fP_*) \cap R \cap \fR_*$.  Since 
    $\Leb(T^{-n}(\fP_*) \cap R \cap \fR_*)>0$ and $\cM_n$ is finite, 
    there must exist at least two stable leaves in $\fS^s_*$ belonging
    to the same component of $\cM_n$ such that
    the solid s-subrectangle $\fR'$ of $\fR$ bounded by these two leaves has $\Leb(\fR')>0$. 
    Then $\fR'$ is a thick s-subrectangle of $\fR$ which s-crosses $R$, on which $T^n$ is continuous, and such that
    $T^n (\fR' \cap \fR_*) \cap \fP_* \neq \emptyset$. As in~\cite[Lemma~7.90]{CM06},
    for sufficiently large $n$ such a crossing implies that $T^n (\fR')$
    u-crosses both $R$ and $\fR$.

    It remains to show that there exists a stable leaf $W$ fully crossing both $R$ and $\fR$
    with $T^n(W) \subset W$. Let $n$ and $\fR'$ be as above. Define $A_0 = \fR'$ and inductively
    $A_{k+1} = A_k \cap T^{-n} (A_k)$. Observe that $A_k$ is a nested sequence of exponentially
    shrinking s-subrectangles of $\fR$ that also s-cross $R$. Its limit $W = \lim_{k \to \infty} A_k$
    is necessarily a stable leaf which s-crosses $R$ and $\fR$, with $T^n(W) \subset W$,
    as required.
\end{proof}

Note that if we find a stable leaf $W$ and $n \ge 1$ such that $T^n(W) \subset W$ as in 
Lemma~\ref{lem:CM-R}, then $W$ is necessarily preperiodic.

\begin{cor}
    \label{cor:R}
    There exists a thick rectangle $R$ such that:
    \begin{enumerate}[label=(\alph*)]
        \item\label{cor:R:0}
            $d(R, \cS_1) > 0$.
        \item\label{cor:R:1}
            $T(R) \cap R = \emptyset$.
        \item\label{cor:R:a}
            The stable leaves of $\partial_s R$ are preperiodic,
            $T^n(\partial_s R) \cap \mathring{R} = \emptyset$ for all $n \geq 0$
            and $\inf_{n \geq 0} d(T^n(\partial_s R), \cS_1) > 0$.
        \item\label{cor:R:b}
            There is $\delta_R > 0$ such that for every $\delta > 0$
            there is $n_\delta \geq 1$ so that every unstable leaf $W$ with $|W| \geq \delta$
            has a subleaf $W'$ on which $T^{n_\delta}$ is continuous,
            $T^{n_\delta}(W')$ fully crosses $R$ and extends at least distance
            $\delta_R$ outside of $R$ on each side.
        \item\label{cor:R:c}
            There are two s-subrectangles $R', R''$ of $R$ and $n',n'' \geq 1$
            with $\gcd \{n', n''\} = 1$ such that:
            \begin{itemize}
                \item $T^{n'}$ is continuous on $R'$,
                \item $T^{n'}(R')$ fully u-crosses $R$,
                \item $d(\partial_u T^{n'}(R'), \partial_u R) > 0$,
                \item each unstable leaf fully crossing $T^{n'}(R')$ extends at least distance $\delta_R$
                    outside of $R$ on each side.
            \end{itemize}
            And the same conditions hold for $R''$ and $n''$.
    \end{enumerate}
\end{cor}

\begin{proof}
    Starting from a thick rectangle $\fR$, chosen away from $\cS_1$
    in an area which guarantees $T(\fR) \cap \fR = \emptyset$
    and using Lemma~\ref{lem:CM-R}, we find two preperiodic stable leaves
    $W, W'$ which fully cross $\fR$ and do not belong to $\partial_s \fR$,
    and so that the s-subrectangle $\fR'$ bounded by $W$ and $W'$ is thick.

    The forward images of $W, W'$ are contained in some
    finite set of preperiodic maximal stable leaves $\{W_k\}$.
    This allows us to find a thick rectangle $R$ contained in the part of $\fR$ between $W$ and $W'$
    such that $\partial_s R \subset \cup_k W_k$ with $R$ a positive distance from the endpoints of $\{W_k\}$,
    and such that $\mathring{R}$ does not intersect any $W_k$.
    Our choice of $R$ implies~\ref{cor:R:0}, \ref{cor:R:1} and~\ref{cor:R:a}.
    
    For~\ref{cor:R:b}, there exists a u-subrectangle $R^\diamond$ of $R$ with
    $\Leb(R^\diamond \cap \fR_*) > 0$ and $d(\partial_u R^\diamond, \partial_u R) > 0$.
    Given an unstable leaf $W$ with $|W| \geq \delta$, we apply
    Lemma~\ref{lem:CM-W} to the rectangle $\fR$ to obtain
    $W'' \subset W$ such that $T^{n'_\delta}(W'')$ u-crosses $\fR$.
    Then apply Lemma~\ref{lem:CM-R} to obtain an s-subrectangle $\fR'$ of $\fR$ so that
    $T^{n_R}(\fR' \cap T^{n'_\delta}(W''))$ u-crosses both $R^\diamond$ and $\fR$.  This yields a curve
    $W' \subset W''$ satisfying the required property with $n_\delta = n_R + n'_\delta$ and 
    $\delta_R \ge d(\partial_s R, \partial_s \fR)$.
    
    Finally, item~\ref{cor:R:c} follows by twice applying Lemma~\ref{lem:CM-R} to $\fR$
    and the u-subrectangle $R^\diamond$ of $R$ as constructed above.
    The overlap $\delta_R$ is the same as in~\ref{cor:R:b}.
\end{proof}

\begin{lemma}
    \label{lem:primeperiodic}
    Suppose $p \geq 2$ is prime and $\gamma \subset M \setminus \cS_{-p}$ is a closed unstable leaf
    such that $T^{-1}(\gamma) \cap \gamma = \emptyset$ and $T^{-p}(\gamma) \subset \gamma$.
    Then there is a neighborhood $U$ of $\gamma$ such that $\min \{n \geq 1 : U \cap T^n(U) \neq \emptyset \} = p$.
\end{lemma}

\begin{proof}
    For $\eps > 0$ let $U_\eps$ denote the closed $\eps$-neighborhood of $\gamma$ and let
    \[
        E_\eps = \{ x \in U_\eps : T^n(x) \in U_\eps \text{ for some } 0 < n < p \}
        .
    \]
    Suppose that the result is wrong. Restrict to $\eps$ sufficiently small so that $T^{-p}$ is continuous on $U_\eps$.
    Then $E_\eps$ are nested non-empty compact sets and thus $E = \cap_{\eps > 0} E_\eps \subset \gamma$ is non-empty,
    so there is $x \in \gamma$ with $T^n(x) \in \gamma$ with some $0 < n < p$.

    Let $\tilde{\gamma}$ be the maximal unstable leaf containing $\gamma$. Then
    $T^{-n}(\tilde{\gamma}) \subset \tilde{\gamma}$ and $T^{-p}(\tilde{\gamma}) \subset \tilde{\gamma}$.
    Since $n$ and $p$ are coprime, $\tilde{\gamma}$ (and hence $\gamma$) has a fixed point.
    This contradicts our assumption that $T^{-1}(\gamma) \cap \gamma = \emptyset$.
\end{proof}

Now we are ready to construct the rectangle $D$ for~(\ref{A:D}).
Let $R$ be as in Corollary~\ref{cor:R} and let $R', R'', n', n''$ be from Corollary~\ref{cor:R}\ref{cor:R:c}.
Let $p$ be a large prime (specified later in~\eqref{eq:p}), and let $p = k' n' + k'' n''$ with $k',k'' \geq 1$.

Construct $D'$ as a u-subrectangle of $R$ obtained by following the itineraries of $R'$ and $R''$
with multiplicities $k'$ and $k''$ steps respectively.
That is, let $D' = D'_{k' + k''}$ where
$D'_0 = R$, $D'_\ell = T^{n'} ( D'_{\ell-1} \cap R' )$ for $1 \leq \ell \leq k'$ and
$D'_\ell = T^{n''} ( D'_{\ell-1} \cap R'')$ for $k'+1 \leq \ell \leq k' + k''$.

By construction, $D'$ is a thin u-subrectangle of $R$ with $T^{-p}$ continuous on $D'$.
Moreover, $T^p(D' \cap T^{-p}(D'))$ is a u-subrectangle of $D'$ fully crossing $D'$.
Set $D_0 = D'$ and $D_\ell = T^p (D_{\ell-1} \cap T^{-p}(D'))$ for $\ell \geq 1$.
Note that $D_\ell \subset D_{\ell - 1}$.

Then $\lim_{\ell \to \infty} D_\ell =: \gamma$ is an unstable leaf
that fully crosses $D'$ and is $p$-periodic in the sense that $T^{-p}(\gamma) \subset \gamma$.
Using Lemma~\ref{lem:primeperiodic}, choose $L$ large so that $\min \{ n \geq 1 : T^n(D_L) \cap D_L \neq \emptyset \} = p$.
Define $D = D_L$. Note that $T^p (D \cap T^{-p}(D'))$ is a u-subrectangle of $D$.

Let $c_0 = d ( \Orb_{n'}(R') \cup \Orb_{n''}(R'') , \cS_0 ) > 0$,
where $\Orb_n(A) = A \cup T(A) \cup \cdots \cup T^n(A)$. Note that $c_0$ is positive and independent of our choice of $D$.
Since $T$ has bounded distortion away from $\cS_0$, \cite[Section~5.6]{CM06}
there is $C_d > 0$, depending only on $c_0$ and the billiard table, such that
$\frac{J_VT^p(x)}{J_VT^p(y)} \le C_d$ for any stable leaf $V \subset T^{-p}(D')$ and all $x, y \in V$, where
$J_VT^p$ denotes the Jacobian of $T^p$ along $V$.
Thus the way $T^p(D \cap T^{-p}(D'))$ sits inside $D$ is comparable to the way $D'$ sits inside $R$.
Moreover, since $D' \subset T^{n''}(R'')$, we have,
\begin{equation}
    \label{eq:DR}
    \frac{d\bigl( \partial_u T^p(D \cap T^{-p}(D')), \partial_u D \bigr)}{\diam_s D}
    \geq C_R
    , \quad \text{where} \quad
    C_R = C_d \frac{d(\partial_u T^{n''}(R''), \partial_u R)}{\diam_s R}
   .
\end{equation}
Here $\diam_s(A)$ is the maximal length of a stable curve contained in $A$.
Informally, $T^p(D \cap T^{-p}(D'))$ sits inside $D$ at least as deeply as $R''$ sits inside $R$.

Recall that by~\eqref{eq:hyp}, if $\theta$ is a stable leaf contained in $D$,
then $|T^p(\theta)| \leq C_e^{-1} \Lambda^{-p} \diam_s D$. We choose $p$ sufficiently large so that
\begin{equation}
    \label{eq:p}
    C_e \Lambda^{-p} \leq C_R / 2
\end{equation}
and let
\[
    \delta_3 = C_R \diam_s D
    , \quad \text{so} \quad
    d ( \partial_u T^p(D \cap T^{-p}(D')), \partial_u D ) \geq \delta_3
    .
\]

We constructed $D$, and it remains to verify that it satisfies the requirements of~(\ref{A:D}).

Indeed,~(\ref{A:D-B}) and~(\ref{A:D-period}) hold by construction. Since $D$ u-crosses $R$,
the minimal length of an unstable leaf fully crossing $D$ is some $\delta_2$ for~(\ref{A:D-long}).

Choosing $L$ (and or $p$) large we make $D$ sufficiently
narrow in the stable direction, as required for~(\ref{A:D-narrow});
the part of~(\ref{A:D-narrow}) related to $\partial_s D$ is satisfied
for sufficiently small $b_0$ because the stable leaves $\partial_s R$ are preperiodic
and $\inf_{n \geq 0} d(T^n(\partial_s R), \cS_1) > 0$.

Now we verify~(\ref{A:D-crossing}).
We say that an unstable leaf $W$ is $n$-good if $T^n$ is continuous on $W$,
$T^n(W)$ properly crosses $R$, and
$B_0 \cap T^k(W) = \emptyset$ for all $0 \leq k \leq n$, where $B_0$ is 
the $b_0$-neighborhood of $\cS_1$ from Definition~\ref{def:b0}.

Apply Corollary~\ref{cor:R}\ref{cor:R:b} with $\delta = \min \{\delta_2, \delta_1 / 3\}$
and get the corresponding $n_\delta$.
Forcing $b_0$ to be sufficiently small and using the bound~\eqref{A:super-growth},
for every unstable leaf $W \subset B_0$
we ensure that $|T^k(W)| < \delta_R$ for all $0 \leq k \leq n_*$, where $n_* = \max \{n_\delta, n', n''\}$.
Then for every unstable leaf $W$ which fully crosses $R$ and extends distance $\delta_R$ outside $R$
on each side, we have $B_0 \cap T^{-k}(W \cap R) = \emptyset$ for all $0 \leq k \leq n_*$,
and in particular $T^{-n}(W)$ has an $n$-good subleaf for each $0 \leq n \leq n_*$.

Then every unstable leaf $W$ with $|W| \geq \delta$ has an $n_\delta$-good subleaf $W_1$.
A further subleaf $T^{-n_\delta} ( T^{n_\delta}(W_1) \cap T^{-n'}(R'))$ is automatically $(n_\delta + n')$-good,
and similarly if we use $R'', n''$ instead of $R',n'$.
Taking a suitable sequence of subleaves, and using that $\gcd \{n', n''\} = 1$,
we can obtain an $n$-good subleaf $W_n \subset W$ for every sufficiently large $n$.
A further subleaf $T^{-n} (T^n(W) \cap T^{-(p + 1) L}(D))$ is $(n + (p + 1) L)$-good and
its image under $T^{n + (p + 1) L}$ fully crosses not only $D$ but its thinner u-subrectangle
$T^p(D \cap T^{-p}(D'))$. With our choice of $\delta_3$ we have verified~(\ref{A:D-crossing}).

Finally,~\eqref{A:D-rD} follows from our choice of $p$ and $\delta_3$.


\subsection{Tree of unstable leaves}

In this section we use~(\ref{A:hprime})--(\ref{A:D}) to construct a tree of unstable leaves
which will serve as a basis for a symbolic model of the dynamics.
We start with an unstable leaf $\wroot$ which fully crosses $D$ and terminates at
the stable boundaries of $D$.

\begin{defn}
    We say that an unstable leaf $W$ \emph{properly crosses} $D$ if $W$ fully crosses $D$ and
    $d(W, \partial_u D) \geq \delta_3 / 2$.
\end{defn}

\begin{defn}
    For the map $T$, a \emph{cylinder} of length $n$ is a domain of continuity of $T^n$
    (an element of $\cM_n$),
    in other words the set of points which follow a given topological itinerary for $n$ steps.
    A concatenation $AB$ of cylinders $A$ and $B$ of lengths $n$ and $m$ respectively
    is $A \cap T^{-n}(B)$, a (possibly empty) cylinder of length $n+m$ which can be thought of as
    the set of points which successively follow topological itineraries of $A$ and $B$.
\end{defn}

For an unstable leaf $W$ and $n \geq 0$, let $A_W^{-n}$ denote the cylinder of length $n$
containing $T^{-n}(W)$, and let $A_W^n = T^n(A_W^{-n})$ be the cylinder of length $n$ for $T^{-1}$
containing $W$.

Recall that $B_n$ denotes the $b_0 \Lambda^{-n}$-neighborhood of $\cS_1$.
Let $\partial B_n$ denote its boundary.

We construct a tree $\bW$, where nodes are unstable leaves.
The root of the tree is $\wroot$; it is the only node of height $0$.
We use $\bW_n$ to denote the set of nodes at height $n$, $n \ge 0$.
For $W \in \bW_n$, its children are obtained as follows.
We apply $T$ to $W$ and erase the intersections of $T(W)$ with $B_{n+1}$.
Each leaf which properly crosses $D$, we cut at the stable boundaries of $D$, creating at most
three components.
The resulting leaves become children of $W$ and elements of $\bW_{n+1}$.

Some nodes in $W \in \bW_n$, $n \geq 1$, we designate as \emph{prime}:
\begin{itemize}
    \item If $W$ properly crosses $D$ and has no prime ancestor, then $W$ is prime.
    \item If $W$ properly crosses $D$ and has a prime ancestor,
        let $n-k$ be the height of its most recent prime ancestor (with smallest $k$).
        If there is no node in $U \in \bW_k$ such that $D \cap A_W^k \cap U \neq \emptyset$,
        then $W$ is prime. (See Figure~\ref{fig:prime def} for two ways in which this may occur.)
    \item Otherwise $W$ is not prime.
\end{itemize}

\begin{figure}[ht]
    \begin{tikzpicture}[x=8mm,y=8mm]

	\draw (0,0) rectangle (6,4);
        \draw (0,1) to[out=10, in=180] (2,1.2) to[out=0, in=180] (4,1.1) to [out=0, in=190] (6,1.2);
        \draw[thick] (0,2.4) to[out=-10, in=180] (2.2, 2.2) to[out=0, in=180] (3.8,2.4) to[out=0, in=170] (6,2.2);

	\draw[dashed] (4.0,2.15) to[out=-55, in=120] (5.0,0.8) to[out=-60, in=100] (5.4,-0.2); 
        \draw[dashed] (3.5,4.2) to[out=-80, in=90] (3.8, 3) to[out=-90, in=100] (4.0,2.15) to[out=-70, in=100] (4.5,-0.2);

	\draw[dashed] (1.9,3.38) to[out=-70, in=100] (2.8,1.8) to[out=-70, in=100] (3.1,0.8) to[out=-80, in=110] (3.3,-0.2); 
        \draw[dashed] (1.8, 4.2) to[out=-90, in=100] (1.9,3.38) to (2.0,2.5) to[out=-90, in=100] (2.2,0.8) to[out=-75, in=105] (2.3,-0.2);

	\node at (3,3.5){\small $D$};
	\node at (1,2.6){\small $\wroot$};
	\node at (2.4,1.9){\small $U'$};
	\node at (0.7,0.7){\small $W'$};
	\node at (2.6,0.8){\small $V'_2$};
	\node at (4.65,0.7){\small $V'_1$};
	\node at (3.5,4.4){\small $\cS_k$};
	\node at (1.8,4.4){\small $\cS_k$};
	
	\node at (3, -0.7){\small $(a)$};

        \draw (10,0) rectangle (16,4);
        \draw (10,1.5) to[out=10, in=180] (12,1.7) to[out=0, in=180] (14,1.5) to [out=0, in=190] (16,1.7);
        \draw[dashed] (9.8,1.7) to[out=5, in=190] (11.5,1.9) to[out=10, in=200] (16.7,2.6) to[out=20, in=195] (18.3,3.1);
        \draw (17,2.5) to[out=-10, in=180] (18,2.55) to[out=0, in=190] (19,2.6);

        \node at (13,1.2){\small $W$};
        \node at (18, 2.2){\small $U$};
        \node at (17,3.2){\small $\cS_{-k}$};
        \node at (13, 3.5){\small $D$};
        
        \node at (14,-0.7){\small $(b)$};

    \end{tikzpicture}
    \caption{
        (a) The most recent prime ancestor $W' \in \bW_{n-k}$ of $W \in \bW_n$ properly crossing $D$ shown with
        singularity curves in $\cS_k$. If $T^{-k}(W) \subset V_1'$ then $A^k_W \cap T^k(\wroot) = \emptyset$.
        On the other hand, if $T^{-k}(W) \subset V_2'$ and there is $U \in \bW_k$ with $T^{-k}(U) \subset U'$,
        then $A^k_W \cap U \neq \emptyset$, but it may happen that $D \cap A^k_W \cap U = \emptyset$.
        (b) Relative positions of a prime node $W$, and $U \in \bW_k$ with $A^k_W \cap U \neq \emptyset$,
        but $D \cap A^k_W \cap U = \emptyset$.
    }
    \label{fig:prime def}
\end{figure}
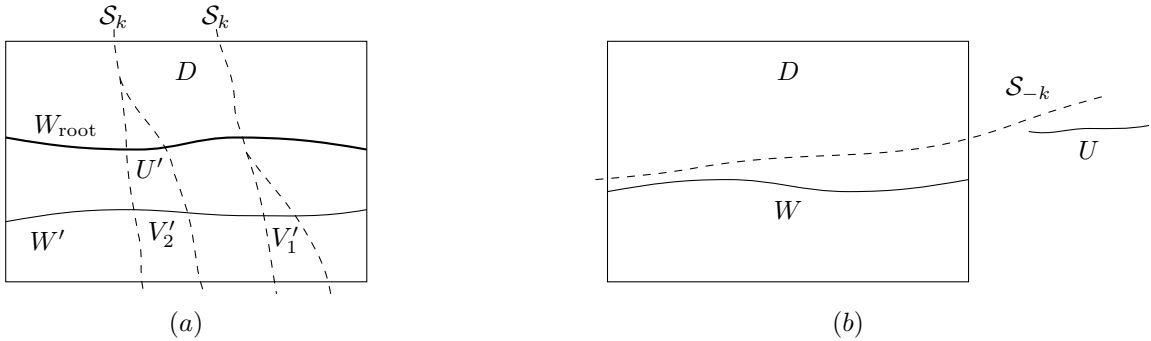

\begin{rmk}
    By construction, and in particular by the choice of $D$ such that $D \cap B_0 = \emptyset$,
    each $W \in \bW$ belongs to $M \setminus \cS_1$, thus $T$ is continuous on $W$.
\end{rmk}

Let $\cR$ be the collection of cylinders associated with all prime nodes:
\[
    \cR = \bigcup_{n=1}^\infty \{ A \in \cM_n : \mbox{$T^n(A)$ contains a prime node of height $n$} \}
    .
\]
Let $\cA$ be the collection of cylinders obtained by taking all finite concatenations of elements of $\cR$.
Let $\cR_n \subset \cR$ and $\cA_n \subset \cA$ be the subcollections of cylinders of length $n$.

\begin{rmk}
    \label{rmk:unique}
    The correspondence between $W \in \bW_n$ and cylinders $A_W^n$ is not one-to-one:
    due to the artificial cuts at $\partial_s D$, two or more nodes can correspond
    to the same cylinder.
\end{rmk}

\begin{rmk}
    \label{rmk:DR}
    Suppose that $R \in \cR_n$ and let $D_R = D \cap R \cap T^{-n}(D)$.
    Our construction ensures that $D_R$ is an s-subrectangle of $D$
    and that $T^n(D_R)$ is a u-subrectangle of $D$.
    The same then holds for $A \in \cA_n$.
\end{rmk}

\begin{defn}
    If $W \in \bW_n$ properly crosses $D$ and $A_W^{-n} \in \cA_n$, then we call 
    $W$ a {\em return node}.  Note that all prime nodes are necessarily return nodes.
\end{defn}

Recall $h' \in (0, h - s_0 \log 4)$ from~\eqref{A:hprime} and let
\begin{equation}
\label{eq:alpha def}
    \alpha = \frac{h - h'}{s_0  \log 2}
    .
\end{equation}
Observe that $\alpha >2$ under the assumption of Theorem~\ref{thm:decay}.

The crux of our argument is the following three propositions.

\begin{prop}
    \label{prop:prime}
    Every cylinder in $\cA$ has a \emph{unique} representation as a concatenation
    of cylinders in $\cR$.
\end{prop}

\begin{prop}
    \label{prop:YCb}
    There exists $C > 0$ such that for all sufficiently large $n$,
    \[
        C^{-1} e^{h n} \le \# \cA_n \le C e^{h n}
        .
    \]
\end{prop}

\begin{prop}
    \label{prop:YTb}
    There exists $C>0$ such that $\# \cR_n \leq C e^{h n} / n^\alpha$ for all $n \ge 1$.
\end{prop}

\begin{rmk}
    The overall number of cylinders of length $n$ grows as $e^{h n}$ up to a multiplicative constant
    by \eqref{eq:exact}.
    Informally, Proposition~\ref{prop:YCb} shows that $\cA$ ``sees'' the full topological pressure,
    Proposition~\ref{prop:YTb} gives a weak bound on the pressure at infinity, and
    Proposition~\ref{prop:prime} shows that cylinders in $\cR$ are
    a prime basis of $\cA$.
\end{rmk}

\begin{rmk}
    In fact, Propositions~\ref{prop:YCb} and~\ref{prop:YTb} hold as well for all $h > s_0 \log 2$
    (i.e.\ under the sparse recurrence condition~\eqref{eq:sparse}) and $h'$ small enough so that $\alpha>1$.
    However, we work with $\alpha>2$ in light of the rate of decay of correlations specified
    by Theorem~\ref{thm:decay}.
\end{rmk}
    
In the remainder of this section we prove Propositions~\ref{prop:prime},~\ref{prop:YCb} and~\ref{prop:YTb}.
We use $C$ to denote various constants which depend continuously
(only) on the constants from (\ref{A:hprime})--(\ref{A:D})
as well as $\cK_{\min}$, $\cK_{\max}$, $\tau_{\min}$ from~\eqref{eq:cU},
$\Lambda$ and $C_e$ from~\eqref{eq:hyp}.

\subsection{Proof of Proposition~\ref{prop:prime}}

\begin{lemma}
    \label{lem:AW}
    For every $A \in \cA_n$ there is a unique $W_A \in \bW_n$ which fully crosses $D$ and
    such that $W_A \subset T^n(A)$.
    Explicily, $W_A = D \cap T^{n}(A) \cap T^n (\wroot)$.
\end{lemma}

\begin{proof}
    The uniqueness of $W_A$ is immediate if it exists: there cannot be two different
    elements of $\bW_n$ fully crossing $D$ and coming from the same cylinder of length $n$.

    The statement is clear if $A \in \cR$.
    Suppose that it holds for $A' \in \cA$ and $R \in \cR$. We show that it holds for $A = A' R$;
    this proves the result in general by induction.

    Suppose that $R$ has length $r$ and $A'$ has length $n'$.
    Let $D_R = D \cap R \cap T^{-r}(D)$ be the s-subrectangle of $D$ as in Remark~\ref{rmk:DR}.
    Then $W_{A'}$ fully crosses $D_R$, and by the construction of $\bW$
    it is enough to check that for each endpoint $x$ of $W_{A'} \cap D_R$,
    \begin{equation}
        \label{eq:tdk}
        d(T^k(x), \cS_1)
        \geq b_0 \Lambda^{-(k + n')}
        \quad \text{for all} \quad
        0 \leq k \leq r
        .
    \end{equation}
    Indeed, let $\gamma$ be the stable leaf on the boundary of $D_R$ containing $x$, and let
    $y = \gamma \cap \wroot$.
    Since $R$ is prime with length $r$, for $0 \leq k \leq r$,
    \[
        d(T^k(y), \cS_1)
        \geq b_0 \Lambda^{-k}
    \]
    by the construction of $\bW$, and 
    \[
        d(T^k(x),T^k(y))
        \leq |T^k(\gamma)| \leq b_0 \Lambda^{-k} / 2
    \]
    by~\eqref{A:D-narrow}.
    Now,~\eqref{eq:tdk} follows from the triangle inequality, $n' \geq r_D$ and~\eqref{A:D-rD}
    since $\Lambda^{-n'} \le 1/2$.
\end{proof}

\begin{proof}[Proof of Proposition~\ref{prop:prime}]
    Suppose that $R A = R' A'$ with $R, R' \in \cR$, $A, A' \in \cA$ and $R \neq R'$.
    We show that this is impossible, which implies the desired result.

    Denote the lengths of $R, R'$ and $A, A'$ by $r,r'$ and $n,n'$ respectively.
    It is impossible that $r= r'$, so we suppose that $r < r'$.

    Using Lemma~\ref{lem:AW}, let $W_{R}, W_{R'}, W_A, W_{A'}, W_{R A}$ and $W_{R' A'}$
    be the nodes of $\bW$ corresponding to the respective cylinders.
    Observe that $W_{R A}$ is a descendant of $W_{R}$ and $W_{R' A'}$ is a descendant of
    $W_{R'}$. At the same time, $R A = R' A'$ implies $W_{R A} = W_{R' A'}$,
    so $W_{R'}$ is a descendant of $W_R$.

    Now, $T^{-n}(W_A) \subset \wroot$ and $T^{-n}(W_{R A}) \subset W_R$ are
    unstable leaves contained in $A \cap D$ whose images under $T^n$ fully cross $D$.
    In particular, there is a stable leaf $\gamma \subset D$ which connects interior points of
    $T^{-n}(W_A)$ and $T^{-n} (W_{R A})$.

    Let $U \in \bW_{r + k}$ be the first prime descendant of $W_R$ on the line of ancestors of $W_{RA}$.
    Since $W_{R'}$ is prime, $U$ is well defined and $k \leq r' - r \leq n$.
    Since $U$ is prime, there is no element of $\bW_k$ which intersects $D \cap A_U^k$.
    At the same time, $U$ is an ancestor of $W_{R A}$, so $U$ intersects $T^k(\gamma)$.
    Moreover, $T^k(\gamma) \subset D$ due to~\eqref{A:D-crossing},~\eqref{A:D-rD}
    and $k \geq r_D$.
    This is a contradiction because $T^k(\gamma)$ contains a point of a node of $\bW_k$.
\end{proof}

\subsection{Proof of Proposition~\ref{prop:YCb}}

We will find it convenient to record the following consequence of \eqref{A:hprime}.
\begin{lemma}
    \label{lem:extend F1}
    Let $W$ be an unstable leaf. Then
    \[
        \max \{|U| : U \in \cG_k(W), \; 0 \leq k \leq n \} \leq \delta'
        \qquad \text{implies} \qquad
        \# \cG_{n}(W) \leq e^{(n + N') h'}
        .
    \]
\end{lemma}

\begin{proof}
    Clearly, \eqref{A:hprime} implies $\# \cG_{n}(W) \leq e^{n h'}$ if $n$ is a multiple of $N'$.
    We extend it to general $n$ by writing $\# \cG_{n}(W) \leq \# \cG_{m}(W)$
    where $m$ is a multiple of $N'$ with $n \leq m < n + N'$.
\end{proof}

For an unstable leaf $W$, a subset $X \subset M$ and a point $x \in M$, let
\[
    \Delta_W^X (x)
    = \{ y \in W : \text{there exists a stable curve } \gamma \subset X \text{ with } x,y \in \gamma \}
    .
\]
For $E \subset M$ let $\Delta_W^X(E) = \cup_{x \in E} \Delta_W^X(x)$.
Informally, $\Delta_W^X(E)$ represents the projection of $E$ onto $W$ along stable curves in $X$.

\begin{rmk}
    Preimages of stable curves, on which $T^{-1}$ is continuous, are stable curves, so 
    if $T$ is continuous on $X$, then
    \begin{equation}
        \label{eq:DeltaT}
    \Delta_{T(W)}^{T(X)} (T(x))
        \subset T (\Delta_W^X(x))
        .
    \end{equation}
\end{rmk}

Denote
\begin{equation}
    \label{eq:En}
    E_n = \wroot \cap T^{-n} \partial B_n
    .
\end{equation}
These are points on $\wroot$ where cuts may occur in the construction of $\bW_n$
when we discard the pieces which fall in $B_n$.

\begin{defn}
    We say that an unstable leaf $W$ crosses $D$ \emph{extra properly} if
    it fully crosses $D$ with $d(W, \partial_u D) \geq \delta_3$.
    We way that $W \in \bW_n$ is \emph{regular} if:
    \begin{enumerate}[label=(\alph*)]
        \item $W$ fully crosses $D$,
        \item $T^n(D) \cap A_W^n$ contains an unstable leaf which crosses $D$ extra properly,
        \item $\Delta_W^{D \cap A_W^k} (T^k(E_j)) = \emptyset$
            for all $0 \leq j \leq k \leq n$.
    \end{enumerate}
\end{defn}

\begin{lemma}
    \label{lem:WB0}
    Suppose that $W \in \bW_n$, $n \geq 1$, is regular. Then $A_W^{-n} \in \cA_n$.
\end{lemma}

\begin{proof}
    The result is trivially true if $W$ is prime, or if $n=1$ because a regular $W \in \bW_1$ is
    necessarily prime.
    The proof continues by induction in $n$: we assume that the result
    holds for all $n < N$ and we aim to prove it for $n = N$.

    Without loss of generality we assume that $W$ is not prime.
    
    Observe that $W$ has a prime ancestor: otherwise $W$ would have to be prime by construction.
    Let $V \in \bW_{n_V}$ be the most recent prime ancestor of $W$ and set $n' = n - n_V$.
    
    Let $U = D \cap A_W^{n'} \cap T^{n'} (\wroot)$.
    Note that $U$ is nonempty, moreover there is $W' \in \bW_{n'}$
    with $W' \cap U \neq \emptyset$: otherwise $W$ would be prime.
    
    We claim that $W'$ is regular.
    Then by the assumption of induction, $A_{W'}^{-n'} \in \cA$ and
    $A_W^{-n} = A_{V}^{-n_V} A_{W'}^{-n'} \in \cA$ as required.

    It remains to verify the claim.

    First we show that
    $\Delta_{W'}^{D \cap A_{W'}^k}(T^k(E_j)) = \emptyset$ for all $0 \leq j \leq k \leq n'$.
    Indeed, if that fails for some $j,k$, then there is a
    stable curve $\gamma \subset D \cap A_{W'}^k$ connecting $W'$ to a point in $T^k(E_j)$.
    Since $W, W' \subset A_{W'}^k = A_W^k$ and $W$ fully crosses $D$, we can extend $\gamma$
    so that it reaches $W$ while staying within $D \cap A_W^k$.
    This contradicts the regularity of $W$, namely $\Delta_W^{D \cap A_W^k} (T^k(E_j)) = \emptyset$.

    In particular, $W' \cap T^{\ell} (\partial B_{n' - \ell}) = \emptyset$ for all $0 \leq \ell \leq n'$.
    Since $W'$ is nonempty and can terminate only at $T^\ell (\partial B_{n' - \ell})$ or $\partial_s D$,
    we conclude that $W'$ fully crosses $D$.

    Using regularity of $W$, let $Q \subset T^n(D) \cap A_W^n$
    be an unstable leaf which crosses $D$ extra properly.
    Let $P$ be the rectangle with unstable boundaries $D \cap W$ and $D \cap Q$.
    Then $T^{-k}(P)$ is a rectangle for all $0 \leq k \leq n$
    and $T^{-n}(P) \subset D$. Moreover, $T^{-n'}(P) \subset D$ because $T^{-n'}(W) \subset V$
    which is a prime node of $\bW_{n'}$.
    Thus $T^{n'} (D) \cap A_{W'}^{n'}$ contains the whole $P$, and in particular $Q \cap D$
    which crosses $D$ extra properly.

    The claim is verified and the proof is complete.
\end{proof}

For an unstable leaf $W$ we denote by $W^{\delta}$ the (possibly empty) part of $W$ obtained
by deleting the $\delta$-neighborhood of each endpoint of $W$.

\begin{lemma}
    \label{lem:Weps}
    Let $\eps, \delta > 0$. For all sufficiently small $b_0$,
    \[
        \# \{ W \in \cG_n(\wroot) : \Delta_{W^\delta}^{D \cap A_W^k}(T^k(E_j)) \neq \emptyset
        \text{ for some } 0 \leq j \leq k \leq n \}
        \leq \eps e^{h n}
        .
    \]
\end{lemma}

\begin{proof}
    In this proof the generic constants $C$ do not depend on $b_0$.
    
    For $x \in M$ denote
    \[
        q_{n,k}(x)
        = \# \{ W \in \cG_n(\wroot) : \Delta_{W^\delta}^{D \cap A_W^k}(x) \neq \emptyset \}
        .
    \]
    If $W$ is as in the right hand side above, then $W$ is contained in the same
    element of $\cM_{-k}$ as $x$. That is, the itinerary of $T^{-j} (W)$, $0 \leq j \leq k$, is fixed by $x$.
    Hence
    \begin{equation}
        \label{eq:qnk}
        q_{n,k}(x)
        \leq \# \cG_{n-k}(\wroot)
        \leq C e^{h (n-k)}
        .
    \end{equation}
    Suppose $W$ is an unstable leaf and $y \in \Delta_W^{D \cap A_W^k}(T^k(E_j))$ with $j \le k$.  Let
    $\gamma$ denote a stable curve in $D \cap A_W^k$ with one endpoint at $y$ and the other in $T^k(E_j)$.
    Then $T^{-k}$ is continuous on $\gamma$ and $T^{-k}(\gamma)$ is still a stable curve with length
    $|T^{-k}(\gamma)| \le C$ and $T^{-(k-j)}(\gamma)$ is a
    stable curve connecting $T^{-(k-j)}(y)$ with $T^j(\wroot) \cap \partial B_j$.  
    It follows from~\eqref{eq:hyp} that
    \[
        |T^{-(k-j)}(\gamma)|
        \le C \Lambda^{-j}
        .
    \]
    On the other hand, since
    $\gamma \subset D$, by~\eqref{A:D-narrow} we have $|\gamma| \le C_e b_0/2$, and so by~\eqref{A:super-growth} applied to stable
    curves, 
    \[
        |T^{-(k-j)}(\gamma)|
        \le C b_0^{2^{-s_0(k-j)}}
        .
    \]
    Putting the two upper bounds on $|T^{-(k-j)}(\gamma)|$ together
    (namely using the inequality $\min \{ A, B \} \le \sqrt{AB}$) 
    and using the triangle inequality, we conclude that
    \[
        d(T^{-(k-j)}(y), \cS_1) \le C b_0^{2^{-s_0(k-j)-1}} \Lambda^{-j/2}
        .
    \]

    Using~\eqref{A:super-growth} again, 
    the distance from $y$ to the closest endpoint of $W$ is at most
    \[
        d_{j,k}
        = C t_{j,k}(b_0) (\Lambda^{-j})^{2^{-s_0 (k-j)-1}}
    \]
    where $t_{j,k}(b_0)$ are some nonnegative functions bounded above by $1$ with 
    $\lim_{b_0 \to 0} t_{j,k}(b_0) \to 0$ for each $j,k$.
    Choose a (sufficiently large) $R > 0$ independent of $b_0$ so that
    \[
        d_{j,k} < \delta
        \ \text{and hence} \ 
        q_{n,k}(T^k (E_j)) = 0
        \quad \text{whenever} \quad
        k - s_0^{-1} \log_2 k + R \leq j \leq k
        .
    \]
    Let $N \geq 1$ large (depending on $\eps$, see below) and choose $b_0$ sufficiently small so that
    \[
        d_{j,k} < \delta
        \ \text{and hence} \ 
        q_{n,k}(T^k(E_j)) = 0
        \quad \text{whenever} \quad
        j \leq k \leq N
        .
    \]
    Using $\# E_j \leq C e^{h j}$ and~\eqref{eq:qnk},
    \[
        \sum_{0 \leq j \leq k \leq n} \sum_{x \in T^k(E_j)} q_{n,k}(x)
        \leq C \sum_{\substack{N < k \leq n  \\  0 \leq j \leq k - s_0^{-1} \log_2 k + R}} e^{h j + h (n-k)}
        \leq C e^{hn} \sum_{N < k \leq n} k^{-\frac{h}{s_0 \log 2}}
        .
    \]
    The constants $C$ do not depend on $b_0$ or $N$, and $h > s_0 \log 2$, thus
    we obtain the desired result by choosing $N$ large (and $b_0$ small).
\end{proof}

\begin{lemma}
    \label{lem:countdiscard}
    Let $\eps, \delta > 0$. For all sufficiently small $b_0$,
    \begin{equation}
        \label{eq:countdiscard}
        \# \{ W \in \cG_n(\wroot) : W^\delta \cap T^{k} (B_{n-k}) \neq \emptyset
            \text { for some } 0 \leq k \leq n
        \}
        \leq \eps e^{hn}
        .
    \end{equation}
\end{lemma}

\begin{proof}
    As in the proof of the previous lemma, here the generic constants $C$ do not depend on $b_0$.
    
    Suppose that $W' \in \cG_{m}(\wroot)$ and denote by $W''$ one connected component of $W' \cap B_{m}$.
    By \eqref{A:transverse}, 
    \begin{equation}
        \label{eq:growth for delta}
        |T^k(W'')|
        \le C_1 |W''|^{2^{-s_0 k}} \le C b_0^{2^{-s_0 k}} \Lambda^{-m2^{-s_0 k}}
        ,
    \end{equation}
    for all $k$ such that the right hand side is less than $\delta'$, 
    where $\delta'>0$ is from \eqref{A:hprime}.
    By~Lemma~\ref{lem:extend F1} and \eqref{A:super-growth},
    $\# \cG_k(W'') \leq e^{h' (k+N')}$ for all $k \leq \frac{1}{s_0} \chi(m, b_0)$, where
    $\chi(m, b_0) = -C + \log_2 (m+1) + \psi(b_0, m)$;
    here $\psi$ is some nonnegative function with $\lim_{b_0 \to 0} \psi(b_0, m) = +\infty$ for each $m$
    but otherwise unimportant.

    Then for $k \leq \frac{1}{s_0} \chi(m, b_0)$,
    \[
        \# \{ W \in \cG_k(W') : W \cap T^{k} (B_{m}) \neq \emptyset \}
        \leq N_{\cS} \# \cG_k(W'')
        \leq C e^{h' k}.
    \]
    Setting $m=n-k$, summing over $W' \in \cG_{n-k}(\wroot)$ and using $\# \cG_{n-k}(\wroot) \leq C e^{h(n-k)}$,
    for $k \leq \frac{1}{s_0} \chi(n-k, b_0)$ we have
    \begin{equation}
        \label{eq:ddr}
        \# \{ W \in \cG_{n}(\wroot) : W \cap T^{k} (B_{n-k}) \neq \emptyset \}
        \leq C e^{h (n-k) + h' k}
        .
    \end{equation}
    Similarly, for $k > \frac{1}{s_0} \chi(n-k, b_0)$, we may apply Lemma~\ref{lem:extend F1} for the first
    $\frac{1}{s_0} \chi(n-k,b_0)$ iterates after an intersection with $B_{n-k}$ and then \eqref{A:upper-exp} after that
    to obtain,
    \begin{equation}
        \label{eq:rrd}
        \begin{aligned}
            \# \{ W \in \cG_n(\wroot) & : W \cap T^{k} (B_{n-k}) \neq \emptyset \}
            \leq C e^{h' \frac{1}{s_0}\chi(n-k, b_0) + h (n - \frac{1}{s_0}\chi(n-k, b_0))}
            \\ &
            = C e^{h n - \frac{h - h'}{s_0} (\log_2(n-k+1) + \psi(b_0, n-k)) }
            = C \frac{e^{h n - \frac{h - h'}{s_0} \psi(b_0, n-k)}}{(n-k+1)^\alpha}
            .
        \end{aligned}
    \end{equation}
    Let $N \geq 0$ be large (depending on $\eps$, see below). Taking a sum of the right hand sides of~\eqref{eq:ddr} and~\eqref{eq:rrd}
    over $N \leq k \leq n$, we have
    \begin{align*}
        \# \{ W \in \cG_{n}(\wroot) & : W \cap T^{k} (B_{n-k}) \neq \emptyset
        \text{ for some } N \leq k \leq n \}
        \\ &
        \leq C e^{h n} \biggl( e^{-(h-h') N}
            + \sum_{m=0}^\infty \frac{e^{- \frac{h-h'}{s_0} \psi(b_0, m)}}{(m+1)^\alpha}
        \biggr)
        .
    \end{align*}
    Recall that $\alpha > 1$, so the series above converges.  
    Now fix $\delta, \ve >0$.  Choose $N$ sufficiently large and $b_0$ sufficiently small so that the
    right hand side above is bounded by $\eps e^{hn}$.
    Finally, using~\eqref{eq:growth for delta}, choose $b_0$ small enough so that the left hand side of~\eqref{eq:countdiscard}
    has no elements with $k \leq N$, and~\eqref{eq:countdiscard} follows.
\end{proof}

\begin{lemma}
    \label{lem:Wreg}
    For each sufficiently small $b_0$ there exist $C_{b_0} > 0$ and $N_{b_0} \geq 1$ such that
    \[
        \# \{ W \in \bW_n : W \text{ is regular } \}
        \geq C_{b_0} e^{h n}
        \quad \text{for all} \quad
        n \geq N_{b_0}
        .
    \]
\end{lemma}

\begin{proof}
    By~\eqref{A:lower-exp}, for all sufficiently large $n$,
    \begin{equation}
        \label{eq:Wdhn}
        \# \{ W \in \cG_n(\wroot) : |W| \geq 3 \delta_1 \}
        \geq C e^{h n}
        .
    \end{equation}
    Applying Lemmas~\ref{lem:Weps} and~\ref{lem:countdiscard} with $\delta = \delta_1$
    and $\eps$ sufficiently small, we see that for all sufficiently large $n$,
    \begin{equation}
        \label{eq:WdhnL}
        \begin{aligned}
            \# \bigl\{ W \in \cG_n(\wroot) : {}
                & |W| \geq 3 \delta_1 ,
                \ W^{\delta_1} \cap T^k(B_{n-k}) = \emptyset \ \text{for all} \ 0 \leq k \leq n ,
                \\
                & \ \Delta_{W^{\delta_1}}^{D \cap A_W^k}(T^k(E_j)) = \emptyset
                    \text{ for all } 0 \leq j \leq k \leq n
            \bigr\}
            \geq C e^{h n}
            .
        \end{aligned}
    \end{equation}
    Since $|W^{\delta_1}| \geq \delta_1$, let $W' \in \cG_{N_2}(W^{\delta_1})$ be as in~\eqref{A:D-crossing}.
    Clearly, $W' \in \cG_{n + N_2}(\wroot)$.
    It is now a direct verification that $W' \in \bW_{n + N_2}$ and $W'$ is regular.
\end{proof}

\begin{proof}[Proof of Proposition~\ref{prop:YCb}]
    The upper bound on $\# \cA_n$ follows from Lemma~\ref{lem:AW} and~(\ref{A:upper-exp}).
    The lower bound follows from Lemmas~\ref{lem:WB0} and~\ref{lem:Wreg}.
\end{proof}

\subsection{Proof of Proposition~\ref{prop:YTb}}

\begin{defn}
    Suppose $W \in \bW_n$. We say that $V \in \bW_{n+k}$ is a \emph{first prime descendant} of $W$ if $V$ is a
    prime descendant of $W$ and there are no other prime nodes between $W$ and $V$.
\end{defn}

\begin{defn}
    For an unstable leaf $U$ write $U \perp \bW_k$ if there is no $Q \in \bW_k$ with $A_U^k = A_Q^k$.
\end{defn}

\begin{lemma}
    \label{lem:FP}
    There exist $h'' \in (0, h)$ and $C'' > 0$ such that the following hold.
    Suppose $W \in \bW_{n+k}$ with $n \geq 0$, $k \ge 1$, and either $n = 0$ or $n \geq 1$, the ancestor of $W$
    at height $n$ is prime and $W \perp \bW_k$.
    Then the number of first prime descendants of $W$ of height $n + k + \ell$
    is bounded by $C'' \exp e^{h'' \ell}$ for all $\ell \geq 0$.
\end{lemma}

\begin{proof}
    We only consider the case $n \geq 1$. The case $n = 0$ is similar and simpler.

    Let $W \in \bW_{n+k}$ be as in the statement of the lemma and let $\bW^W$ denote the subtree of $\bW$
    rooted at $W$. So that, for instance, $\bW^W_j$ are the descendants of $W$ in $\bW_{n+ k + j}$.

    For $U \in \bW^W_\ell$ let $\fG_j(U)$ denote the set of all descendants of $U$ in $\bW^W_{\ell + j}$,
    or $\{U\}$ when $j = 0$. Let $\brfG_j(U) \subset \fG_j(U)$ denote the collection of descendants of
    $U$ which are not prime and do not have an ancestor in $\bW^W$ that is prime, except possibly $W$ itself.
    If $U \subset \bW^W_\ell$ is not an individual node but a collection of nodes,
    define $\fG_j(U) = \cup_{u \in U} \fG_j(u)$ and $\brfG^j(U) = \cup_{u \in U} \brfG_j(u)$.

    We note that for $U \in \bW^W$, a single element $V \in \cG_j(U)$ can contain at most $N_{\cS}' = N_{\cS} + 2 N_1 + 1$ nodes in $\fG_j(U)$.
    Indeed, by the construction of the tree and~\eqref{A:transverse}, $V$ can be cut into at most $N_{\cS} + 1$
    nodes by intersections with $B_{n+k+j}$. In addition, endpoints of nodes can be generated by cuts at $\partial_s D$,
    and by~\eqref{A:D-period}, $V$ intersects $\cup_{\ell \geq 0} T^\ell(\partial_s D)$ in at most $2 N_1$ points.
    
    Hence
    \begin{equation}
        \label{eq:frak upper}
        \# \fG_j(U) \le N_{\cS}' \# \cG_j(U)
        .
    \end{equation}
 
    For $U \subset \bW^W_\ell$, define
    \[
        P(U)
        = \sup_{j \geq 0} e^{-h j} \, \# \fG_j(U)
        .
    \]

    Observe that:
    \begin{enumerate}[label=(\alph*)]
        \item\label{eq:disjoint P} $P(U \cup V) \leq P(U) + P(V)$.
        \item\label{eq:UPU} $\# U \leq P(U) \leq C_1 N_{\cS}' \# U$ by~\eqref{eq:frak upper} and~\eqref{A:upper-exp}.
        \item $P(\fG_j(U)) \le e^{j h} P(U)$.
    \end{enumerate}
       
    Choose $N = j_0 N'$, where $N'$ is from~\eqref{A:hprime} 
    and $j_0 \in \mathbb{N}$ is sufficiently large
    that $C_1 N_{\cS}' e^{(h'-h)N} < 1/2$.  Set $\brN = N + N_2$, where $N_2$ 
    is from \eqref{A:D-crossing}.  Also, using \eqref{A:lower-exp} and Proposition~\ref{prop:YCb}, 
    let $C_3>0$ be such that
    \begin{equation}
        \label{eq:Vcro}
        \text{if $V$ properly crosses } D \text{, then } \fG_j(V) \ge C_3 e^{jh} \text{ for all } j \ge 0
        .
    \end{equation}

    Now we consider two possibilities for a node $U \in \bW^W$:
    \begin{itemize}
        \item For each $V \in \fG_{j N'}(U)$ and $0 \le j \le j_0$, $|V| \le \delta_1$.
            Then applying \eqref{A:hprime} inductively for each $j \le j_0$, we get
            $\# \fG_N(U) \leq e^{h' N} N_{\cS}'$. Write
            \begin{align*}
                P(\fG_{\brN}(U))
                & \leq \sum_{V \in \fG_N(U)} P(\fG_{N_2}(V))
                \leq \# \fG_N(U) e^{h N_2} \sup_{V \in \fG_N(U)} P(V)
                \\
                & \le e^{h' N + h N_2} C_1 N_{\cS}'
                \le \tfrac{1}{2} e^{h \brN} P(U)
                .
            \end{align*}
        \item There exists $i \in [0,j_0]$ and $u \in \fG_{i N'}(U)$ such that $|u| \ge \delta_1$.
            Then by~\eqref{A:D-crossing}, there exists $V \in \fG_{i N'+N_2}(U)$ such that $V$ properly crosses $D$.
            Since $W \perp \bW_k$, $V$ is either a first prime descendant of $W$ or has an ancestor in $\bW^W$ that is such.
            For simplicity, we shall still denote by $V$ the first prime descendant of $U$.
            Subsequently, using~\eqref{eq:Vcro}, we eliminate the subtree starting at $V$ to estimate,
            \[
                \# \fG_j(\brfG_{\brN}(U)) \leq \# \fG_j(\fG_{\brN}(U)) - C_3 e^{j h}
                \quad \text{for all} \quad
                j \geq 0
                .
            \]
            Then using (c),
            \begin{align*}
                P(\brfG_{\brN}(U))
                & = \sup_{j \ge 0} e^{-jh} \# \fG_j(\brfG_{\brN}(U)) 
                \leq \sup_{j \ge 0} e^{-jh} \# \fG_j(\fG_{\brN}(U)) - C_3
                \\
                & = P(\fG_{\brN}(U)) - C_3
                \leq e^{\brN h} P(U)  - C_3
                \leq e^{\brN h} P(U) (1 - e^{- \brN h} C_3)
                .
            \end{align*}
    \end{itemize}

    In either case, with $\delta = \min \{ 1/2, e^{- \brN h} C_3 \}$, for every $U \in \bW^W$ we have
    \begin{equation}
        \label{eq:GGG}
        P(\brfG_{\brN}(U)) \leq (1-\delta) e^{h \brN} P(U)
        .
    \end{equation}
    Iterating~\eqref{eq:GGG} and using~\ref{eq:disjoint P} we see that for all $j \geq 0$,
    \[
        P(\brfG_{j \brN}(W))
        \leq (1-\delta)^j e^{h j \brN} P(W)
        .
    \]
    The result follows by~\ref{eq:UPU} since each first prime descendant of $W$ at time $\ell$ must descend from an
    element of $\brfG_{\ell-1}(W)$.
\end{proof}

\begin{lemma}
    \label{lem:firstprime nperp}
    Suppose that $n \ge 1$ and $W \in \bW_n$ is prime.  For $k \geq 1$ let
    \[
        P_k^{\not\perp}
        = \{ V \in \bW_{n + k} : V \text{ is a first prime descendant of } W \text{ with } V \not\perp \bW_k\}
        .
    \]
    Then $\# P_k^{\not\perp} \leq C e^{ h k } k^{- \frac{h}{s_0 \log_2} }$ for all $k$. In particular,
    $\# P_k^{\not\perp} \leq C e^{hk} k^{-\alpha}$.
\end{lemma}

\begin{proof}
Let $V \in P_k^{\not\perp}$ and let $A^k_V$ denote the element of $\cM_{-k}$ containing $V$.  
By uniform hyperbolicity~\eqref{eq:hyp}, the stable diameter of $A^k_V$ is at most $C \Lambda^{-k}$.
Without loss of generality, we assume $C \Lambda^{-k} < \delta_3 / 2$.
Thus since $V$ is prime, $A^k_V$ crosses $D$ fully in the unstable direction; 
in particular, $D$
divides $A^k_V \setminus D$ into two connected components, one to the left and one to the right of
$D$.
Note that
$\partial A^k_V$ comprises elements of $\cS_{-k}$ and so cannot intersect the unstable manifolds
in $D$.

Since $V \not\perp \bW_k$, there exists $U \in \bW_k$ such that
$U \subset A^k_V$ and, moreover, by definition of a prime node,
$D \cap A^k_V \cap U = \emptyset$.
Without loss of generality, suppose that the relative positioning of $U$ and $V$ is as in
Figure~\ref{fig:prime}(a).

\begin{figure}[ht]
    \begin{tikzpicture}[x=8mm,y=8mm]

        \node at (7,-.7){\small$(a)$}; 
        
        \draw (3,0) rectangle (9,4);
        \node at (6, 3.5){\small $D$};
        \draw (3,1) to[out=10, in=180] (5,1.2) to[out=0, in=180] (7,1) to [out=0, in=190] (9,1.2);
        \node at (6,.7){\small $V$};
        \draw (11,2.35) to[out=-10, in=180] (11.75,2.4) to[out=0, in=190] (12.5,2.45);
        \node at (11.7, 2){\small $U$};
        \node at (11, 2.35) [circle,draw,fill,minimum width=2pt,inner sep=0pt,label={left:{\small $x$}}] {};
        \draw[dashed] (10.8, 0.5) to[in=-100,out=90] (11, 2.35) to[in=-90, out=80] (11.2, 3.5);
        \node at (11, 3.8){\small $\gamma$};

        \node at (17.3,-0.7){\small$(b)$};
        
        \draw[thick] (16,4) to[out=-80, in=90] (16.5,2) to[out=-90, in=100] (17,.5);
        \node at (15.5,2){\small $S \in \cS_1$};

        \draw (17.5,3.5) to[out=0, in=190] (18.5,3.6) to[out=10, in=180] (19.5,3.7);
        \node at (18.8,3.2){\small $U'$};
        \node at (17.5, 3.5) [circle,draw,fill,minimum width=2pt,inner sep=0pt,label={left:{\small $x'$}}] {};

        \draw (17.0, 1.05) to[out=10, in=180] (17.5,1.12);
        \node at (17.4,0.75){\small $V'$};

        \draw[dashed] (17.7, 1.0) to[out=85, in=-85] (17.5, 3.5) to[out=95, in=-90] (17.55, 4.0);
        \node at (17.5, 4.4){\small $\gamma'$};

    \end{tikzpicture}
    \caption{(a) A prime node $V$ properly crossing $D$ and $U \subset T^k(\wroot)$ contained in $A^k_V \setminus D$. 
    (b) The preimages under $T^i$ with $U' = T^{-i}(U)$, $\gamma' = T^{-i}(\gamma)$ and so on.
    Stable curves are approximately vertical and unstable curves are approximately horizontal.}
    \label{fig:prime}
\end{figure}
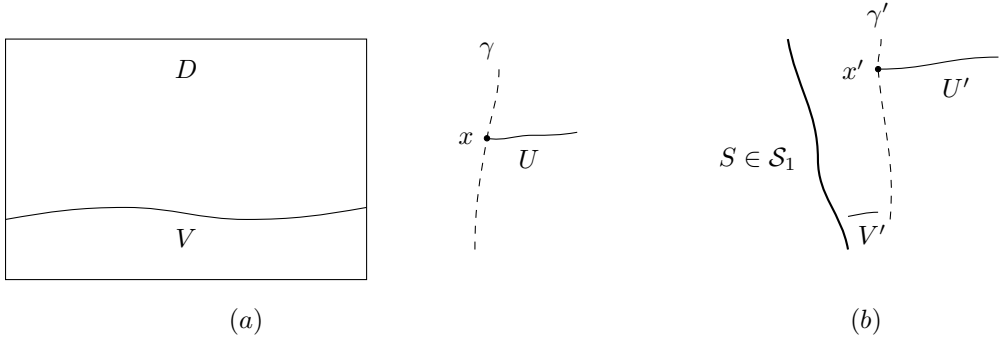

By construction, the left endpoint $x$ of $U$ belongs to $T^{i}(\partial B_{k - i})$ for some $0 < i < k$.
Let $A_i \in \cM_{-i}$ denote the element of $\cM_{-i}$ containing $U$ and $V$.
Let $\gamma \subset A_i$ be a stable curve passing through $x$, fully crossing $A_i$ and not crossing the boundaries of $D$.

Consider the preimages $U' = T^{-i}(U)$, $V' = T^{-i}(V)$, $\gamma' = T^{-i}(\gamma)$ and $x' = T^{-i}(x)$.
Since $T$ is orientation preserving, we may represent them as in
Figure~\ref{fig:prime}(b).
Note that $\gamma'$ is a stable curve fully crossing $T^{-i}(A_i)$.
Since $x \in T^{i}(\partial B_{k - i})$, there is a stable curve $S \in \cS_1$ with $d(x', S) = b_0 \Lambda^{-(k - i)}$.
Using the continuation of singularities property, suppose that $S$ is long in the vertical direction, say that it fully crosses $M$.

Since $U'$ and $V'$ belong to the same element of $\cM_{-(k- i)}$, they are close in the stable direction in the sense that
there is an unstable curve so that both $U'$ and $V'$ are in its $C \Lambda^{-(k - i)}$-neighborhood.
Using transversality of stable and unstable cones, we observe that 
\begin{equation}
    \label{eq:small V}
    |V'|
    \leq C \Lambda^{-(k - i)}
    .
\end{equation}
Indeed, if $|V'| \geq C \Lambda^{-(k - i)}$ with a sufficiently large $C$, then
$V'$ has to cross $\gamma'$ which is impossible by construction.
Note that $V'$ cannot go below or above $\gamma'$ because $V' \subset T^{-i}(A_i)$ which $\gamma'$ fully crosses.

We are ready to estimate $\# P_k^{\not\perp}$.
Each $V \in P_k^{\not\perp}$ is associated with an element $U \in \bW_k$, or more conveniently to its endpoint $x$ as in Figure~\ref{fig:prime}.
Then $x \in T^k(E_{k - i})$ and $|T^{-i}(V)| \leq C \Lambda^{-(k - i)}$ for some $0 < i < k$.
For each $i$ there are at most $\# E_{k - i} \leq C e^{h (k - i)}$ options for $x$, and to each $x$ there corresponds at most one $V$.
Furthermore, by~\eqref{A:super-growth} and~\eqref{eq:small V},
\[
    |V| \leq C |T^{-i}(V)|^{2^{-s_0 i}} \leq C \Lambda^{-(k-i) 2^{-s_0 i}}
    .
\]
Since $V$ is sufficiently long to fully cross $D$,
$|V| \ge \delta_2$ by \eqref{A:D-long},
we necessarily have $(k - i) 2^{- s_0 i} \leq - \frac{\log_2(\delta_2/C)}{\log \Lambda}$,
and hence $i \geq \frac{1}{s_0} \log_2 (k - i) - C$. In particular,
\[
    i
    \geq \min \Bigl\{ j : j \geq \frac{1}{s_0} \log_2 (k - j) - C \Bigr\}
    \geq \frac{1}{s_0} \log_2 k - C
    .
\]
This allows us to bound
\[
   \# P_k^{\not\perp}
   \leq C \sum_{i \geq \frac{1}{s_0} \log_2 k - C} e^{ h (k - i) }
   \leq C e^{ h k - \frac{h}{s_0} \log_2 k }
   = C e^{ h k } k^{- \frac{h}{s_0 \log_2} }
   ,
\]
as required.
\end{proof}

The following estimate will be useful for the subsequent lemma.

\begin{lemma}
    \label{lem:eab}
    Suppose $a > 0$ and $b \in \bR$. There exists $C_{a,b} > 0$, depending only on $a,b$ such that for all $n \geq 1$:
    \[
        \sum_{j = 1}^n e^{a j} j^b
        \leq C_{a,b} e^{a n} n^b
        .
    \]
\end{lemma}

\begin{proof}
    Without loss of generality suppose that $n$ is sufficiently large so that $e^{a j / 2} j^b \leq n^b e^{a n / 2}$
    for all $1 \leq j \leq n$. Then $e^{a j} j^b \leq e^{a n} n^b e^{- a (n - j) / 2}$ and
    \[
        \sum_{j = 1}^n e^{a j} j^b
        \leq e^{a n} n^b \sum_{j=1}^n e^{- a (n - j) / 2}
        \leq e^{a n} n^b (1 - e^{-a/2})^{-1}
        .
    \]
\end{proof}

\begin{lemma}
    \label{lem:firstprime perp}
    Suppose that $W \in \bW_n$ is prime and for $k \geq 1$ let
    \[
        P_k^{\perp}
        = \{ V \in \bW_{n + k} : V \text{ is a first prime descendant of } W \text{ with } V \perp \bW_k\}
        .
    \]
    Then $\# P_k^{\perp} \leq C e^{hk} k^{-\alpha}$ for all $k$.
\end{lemma}

\begin{proof}
    We will call $U \in \bW_{n + \ell}$ a first {\em perp} descendant of $W$ if 
    (a) $U$ is a descendant of $W$, (b) $U \perp \bW_\ell$, and (c) each ancestor of $U$, 
    $U' \in \bW_{n+j}$, $j < \ell$, satisfies $U' \not\perp \bW_j$.
    Since each $V \in P_k^\perp$ is descended from a first perp descendant $U$ of $W$,
    $U \in \bW_{n+\ell}$ for some $\ell \le k$, our goal will be to estimate the number of 
    first perp descendants of $W$ and their first prime descendants.
    
    Let $U \in \bW_{n + \ell}$ be a first perp descendant of $W$    
    and let $U' \in \bW_{n + \ell - 1}$ be the parent of $U$.
    Then $T^{-1}(U) \subset U'$ and $U' \not\perp \bW_{\ell-1}$.

    Let $U'' \in \bW_{n+\ell-2}$ be the parent of $U'$ and consider the curve $T(U'')$. By
    construction, $U'$ is formed from $T(U'')$ after removing intersections of $T(U'')$ with
    $B_{n+\ell-1}$ and possibly cutting at $\partial_s D$.
    Since $U' \not\perp \bW_{\ell-1}$, there exists $Q' \in \bW_{\ell-1}$
    such that $A^{\ell-1}_{Q'} = A^{\ell-1}_{U'}$.  Let $Q'' \in \bW_{\ell-2}$ denote the parent of $Q'$.
    Then $Q'$ is formed from $T(Q'')$ after we remove intersections of $T(Q'')$ with $B_{\ell-1}$ and possibly cut at $\partial_s D$.
    See Figure~\ref{fig:perp}.

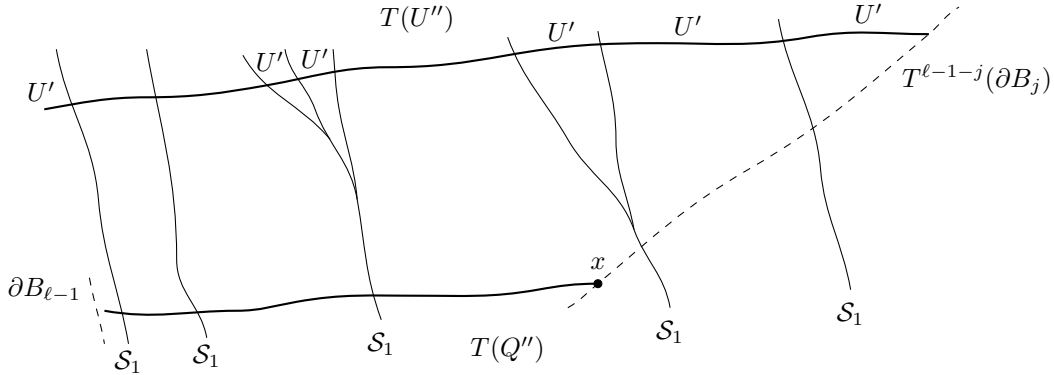
\begin{figure}[ht]
\begin{tikzpicture}[x=8mm,y=8mm]

\draw[thick] (2.3,4.5) to[out=10, in=180] (4,4.7) to[out=0, in=190] (6,4.9) to[out=10, in=180] (8,5.2) to[out=0, in=190] (10,5.4) to[out=10, in=180] (12.5,5.6) to[out=0, in=190] (15,5.7) to[out=10, in=180] (17,5.75);
\draw[thick] (3.3,1.15) to[out=-10, in=180] (5.5,1.15) to[out=0, in=190] (6.5,1.3) to[out=10, in=180] (8.5,1.4) to[out=0, in=190] (10.5,1.5) to[out=10, in=180] (11.5,1.6);

\draw[dashed] (3.05, 1.7) to[out=-80, in=100] (3.3,0.6);
\draw(2.5,5.5) to[out=-80, in=95] (3.2,3) to[out=-85, in=100] (3.7,.6) ;

\draw(4,5.5) to[out=-80, in=90] (4.5,2) to[out=-90, in=100] (5,.7) ;

\draw (6.3, 5.5)  to[out=-80, in=110]  (6.8,4.5) to[out=-70, in=100] (7.5,3)  to[out=-80, in=110]   (7.9,1) ;
\draw (7.1,5.5) to[out=-85, in=100] (7.2, 4.5) to[out=-80, in=95] (7.5, 3);
\draw (5.6, 5.4) to[out=-60, in=130] (7.05, 4);

\draw (10,5.7) to[out=-70, in=120] (11, 4) to[out=-60, in=110] (12.1,2.5) to[out=-70, in=100] (12.7,1.2);
\draw (11.5,5.8) to[out=-80, in=90] (11.8, 4) to[out=-90, in=100] (12.1,2.5);

\draw(14.5,6) to[out=-80, in=95] (15.2,3.5) to[out=-85, in=100] (15.7,1.5) ;

\draw[dashed] (11,1.2) to[out=30, in=220] (11.5,1.6) to[out=40, in=210] (14,3.5) to[out=30, in=220]  (17,5.75) to[out=40, in=210] (17.5,6.2);

\node at (10,.5){\small $T(Q'')$};
\node at (8.5, 6){\small $T(U'')$};
\node at (2.3,1.5){\small $\partial B_{\ell-1}$};
\node at (3.7,.3){\small $\cS_1$};
\node at (5,.4){\small $\cS_1$};
\node at (7.9,.6){\small $\cS_1$};
\node at (12.7,.8){\small $\cS_1$};
\node at (15.7,1.1){\small $\cS_1$};

\node at (11.5, 1.6) [circle,draw,fill,minimum width=3pt,inner sep=0pt,label={above:{\small $x$}}] {};

\node at (17.8,5){\small $T^{\ell-1-j}(\partial B_j)$};

\node at (2.25,4.8){\small $U'$};
\node at (6.05,5.3){\small $U'$};
\node at (6.8,5.35){\small $U'$};
\node at (10.9,5.8){\small $U'$};
\node at (13,5.9){\small $U'$};
\node at (16,6.1){\small $U'$};

 \end{tikzpicture}
\caption{Intersection of $T(U'')$ and $T(Q'')$ with $\cS_1$ and possible locations of $U'$ leading
to $U \perp \bW_\ell$.}
\label{fig:perp}
\end{figure}

Nodes $U'$ that lead to $U \perp \bW_k$ can be formed in two ways.
\begin{itemize}
    \item[(1)] $U' \subset \Delta_{T(U'')}^M (T(Q''))$, i.e.\
        $U'$ can be connected to $T(Q'')$ by a foliation of stable curves.
        Yet $U'$ is necessarily separated from $T(Q'') \setminus B_{\ell-1}$ by curves in $\cS_1$.
        Since curves in $\cS_1$ lie in the
        stable cone, $U' \subset \Delta^M_{T(U'')}(T(Q'') \cap B_{\ell-1})$.
        Both the length of $T(Q'') \cap B_{\ell-1}$ and the distance between $T(U'')$ and $T(Q'')$ along
        stable curves is at most $C b_0 \Lambda^{-\ell + 1}$, so using the transversality of the
        stable and unstable cones,
        \[
            |U'| \le C b_0 \Lambda^{-\ell +1}
            .
        \]
        See for example the three 
        left-most possible $U'$ in Figure~\ref{fig:perp}.
        There are at most $N_{\cS}+1$ possible such $U'$ stemming from $U'' \in \bW_{\ell-2}$
        and at most $C e^{h \ell}$ possible such $U''$.
    \item[(2)] $U' \not\subset \Delta_{T(U'')}^M (T(Q''))$. 
        This can happen
        at one of the two ends of $T(U'')$.  In this case, $U'$ may be long or short, see for example
        the right-most three possible $U'$ in Figure~\ref{fig:perp}.
        Yet in such a case, the corresponding endpoint $x$ of $T(Q'')$ was created at some time
        $j < \ell-1$ by an intersection with $B_j$, i.e.\ $T^{-(\ell-1-j)}(x) \in \partial B_j$
        (and in turn $T^{-(\ell-1)}(x) \in E_j$).  

        Let $\gamma$ be the stable line with maximum slope $-\cK_{\min}$ (see \eqref{eq:cU}) through $x$, contained and fully crossing
        the element of $\cM_{-(\ell-1-j)}$ containing $x$.
        Then no element of $\cS_1$ may both intersect $T(U'')$ to the left of $\gamma$ and pass to the right of $x$.
        Thus $U'$ must lie on the right of $\gamma$.

        \begin{figure}[ht]
            \begin{tikzpicture}[x=8mm,y=8mm]
                \draw[thick] (21.2,4.4) to[out=-90, in=95] (21.6,2) to[out=-85, in=100] (22,.5);
                \node at (22.5,2.5){\small $S \in \cS_1$};

                \draw (17.5,3.5) to[out=0, in=190] (18.5,3.6) to[out=10, in=180] (20.0,3.7);
                \node at (19,4.2){\small $T^{-(\ell - 1 - j)}(U')$};

                \draw (14.0, 1.4) to[out=10, in=180] (18.5,1.5);
                \node at (17.3,1.0){\small $T^{-(\ell - j)}(Q'')$};

                \draw[dashed] (18.7, 1.0) to[out=115, in=-85] (18.5, 1.5) to[out=115, in=-75] (17.2, 4.0);
                \node at (16.5, 2.4){\small $T^{-(\ell - 1 - j)}(\gamma)$};
                \node at (18.5, 1.5) [circle,draw,fill,minimum width=2pt,inner sep=0pt,label={right:{\small $T^{-(\ell - 1 - j)}(x)$}}] {};

            \end{tikzpicture}
            \caption{Preimages of $U'$ and $Q''$ near the singularity curve.}
            \label{fig:Q''}
        \end{figure}
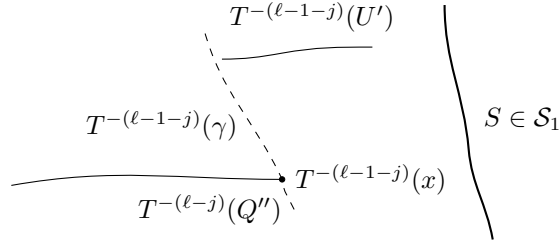

        Considering the preimages of $U'$ and $T(Q'')$ under $T^{\ell - 1 - j}$ as in Figure~\ref{fig:Q''},
        where $T^{-(\ell - 1 - j)}(x)$ is at the distance $b_0 \Lambda^{-j}$ of some stable curve $S \in \cS_1$,
        and following the logic of the proof of Lemma~\ref{lem:firstprime nperp}, we observe that
        \begin{equation}
            \label{eq:TU'}
            |T^{-(\ell-1-j)}(U')|
            \le C b_0 \Lambda^{-j}
            .
        \end{equation}

        There are at most $N_{\cS}+1$ possible $U'$ created for each such endpoint $x \in T^{\ell - 1}(E_j)$,
        and there are at most $\# E_j \leq C e^{h j}$ options for $x$.
\end{itemize}

Our analysis can be summarized as follows: each first perp descendant $U \in \bW_{n + \ell}$ of $W$
is, for some $0 \leq j < \ell$, a child of one of at most $C e^{h j}$ parents $U' \in \bW_{n+\ell-1}$
satisfying~\eqref{eq:TU'}.
Subsequently, $|U| \leq Z_{j,\ell}$ with $Z_{j, \ell} = C ( \Lambda^{-j} )^{2^{-s_0(\ell - j)}}$.

Set $L_{j, \ell} = \max \{ 0, s_0^{-1} \log_2 j - (\ell - j) - R \}$, where $R > 0$ is so large that
if $|U| \leq Z_{j, \ell}$ then every leaf in $\cG_m(U)$, $0 \leq m < L_{j, \ell}$, has length at most $\min\{\delta', \delta_2 / 2\}$.
In particular, $\# \cG_m(U) \leq C e^{h' m}$ and $U$ has no prime descendants in $\bW_{n + \ell + m}$, counting $U$ itself.

For $m > L_{j,\ell}$, either $U$ is prime itself, or the first prime descendants of $U$ are created with rate at most
$C'' e^{h'' (m - L_{j, \ell}) }$ according to Lemma~\ref{lem:FP}. Putting our estimates together, summing over $U$ and
its first prime descendants, we get
\[
    \# P_k^{\perp} 
    \leq C \sum_{j \leq \ell \leq k - L_{j, \ell}} \exp \{ h j + h' L_{j, \ell} + h'' (k - \ell - L_{j, \ell}) \}
    .
\]
Observe that since $h', h'' > 0$, the summand above decreases exponentially with $\ell$.
Hence taking the sum over $\ell \geq j$,
\[
    \# P_k^{\perp}
    \leq C \sum_{j \leq k - L_{j, j}} \exp \{ h j + h' L_{j, j} + h'' (k - j - L_{j, j}) \}
    .
\]
Since $L_{j,j} = s_0^{-1} \log_2 j - R$ for $j \ge 2^{s_0 R}$, we define
\[
    k' = \max \{ j : j + s_0^{-1} \log_2 j \le k \} \le k - s_0^{-1} \log_2 k + C
    ,
\]
so that
\begin{align*}
    \# P_k^{\perp}
       & \leq C e^{k h''} \sum_{j \le k'} \exp \{ (h - h'') j + (h' - h'') s_0^{-1} \log_2 j \}
       \\
       & \leq C \exp \{ k h'' + (h - h'') k' + (h' - h'') s_0^{-1} \log_2 k' \}
       \\
       & \leq C \exp \{ k h - (h - h') s_0^{-1} \log_2 k \}
       = C e^{kh} k^{-\alpha}
       ,
\end{align*}
where for the second inequality we used Lemma~\ref{lem:eab},
and for the third inequality we used $k - k' \leq C + \log_2 k$
and $\log_2 k - \log_2 k' \leq C$.
\end{proof}

\begin{cor}
    \label{cor:firstprime}
    Suppose that $W \in \bW_n$ and either $n = 0$ and $W = \wroot$ or $n \geq 1$ and $W$ is prime.
    Then the number of first prime descendants of $W$ in $\bW_{n + k}$
    is bounded above by $C e^{hk} k^{-\alpha}$.
\end{cor}

\begin{proof}
    The bound for $n = 0$ follows from Lemma~\ref{lem:FP},
    and the bound for $n \geq 1$ follows from Lemmas~\ref{lem:firstprime nperp} and~\ref{lem:firstprime perp}.
\end{proof}

\begin{proof}[Proof of Proposition~\ref{prop:YTb}]
    We will combine our estimates from Corollary~\ref{cor:firstprime} and Proposition~\ref{prop:YCb}
    with the probabilistic argument from Proposition~\ref{prop:ktails} in the appendix.

    Let $\bW_n^r$ be the set of return nodes in $\bW_n$, namely
    \[
        \bW_n^r
        = \{ W \in \bW_n : W \text{ properly crosses } D \text{ and } A_W^{-n} \in \cA_n \}
        .
    \]
    By Proposition~\ref{prop:YCb}, $C^{-1} e^{hn} \leq \# \bW_n^r \leq C e^{hn}$.

    Fix a large $n$ and supply $\bW_n^r$ with the normalized probability counting measure $\bP$.
    For $W \in \bW_n^r$ and $0 \leq j \leq n$, denote by $W_j$ the ancestor of $W$ at height $j$.

    Let $\tau(W) = \# \{1 \leq j \leq n : W_j \; \text{is prime} \}$ and
    define $t_k(W)$ for $0 \leq k \leq \tau(W)$ by $t_0(W) = 0$ and
    $t_{k+1} = \min \{j > t_k : W_j \; \text{is prime} \}$ recursively.
    Set $X_j = (t_j - t_{j-1}) 1_{\tau \geq j}$ and $S = \sum_{j \geq 1} X_j$.

    Suppose $V \in \bW_\ell^r$ with $1 \leq \ell < n$.
    By Proposition~\ref{prop:YCb} and Lemma~\ref{lem:AW},
    the number of elements of $\bW_n^r$ that are descended from $V$ and
    have no prime ancestors at heights $\ell + 1, \ldots, n$ is at least $C e^{h(n-\ell)}$.
    Applying this with $\ell = t_j$ yields
    \[
        \bP \bigl( \tau = j \mid \tau \geq j, X_1, \ldots, X_j \bigr)
        \geq 1 - \theta
        \quad \text{for all} \quad
        j \geq 1
    \]
    with some $\theta \in (0, 1)$ independent of $j$ or $n$.
    Similarly, by Corollary~\ref{cor:firstprime},
    \[
        \bP \bigl( X_{j+1} = k \mid \tau \geq j + 1, X_1, \ldots, X_j \bigr)
        \leq C k^{ - \alpha}
        \quad \text{for all} \quad
        j \geq 0
        ,\ 
        k \geq 1
        .
    \]
    Having verified the hypotheses of Proposition~\ref{prop:ktails}, we conclude that
    $\bP(S = k) \leq C k^{-\alpha}$ for all $k \ge 1$.
    The desired result follows.
\end{proof}


\section{Symbolic model}
\label{sec:symbolic}

Let $r_n$, $n \geq 1$ be a sequence of nonnegative integers, not all of them $0$.
Consider a directed graph as on the picture below with $r_n$ arrows going from
$\circled{n}$ to $\circled{1}$.
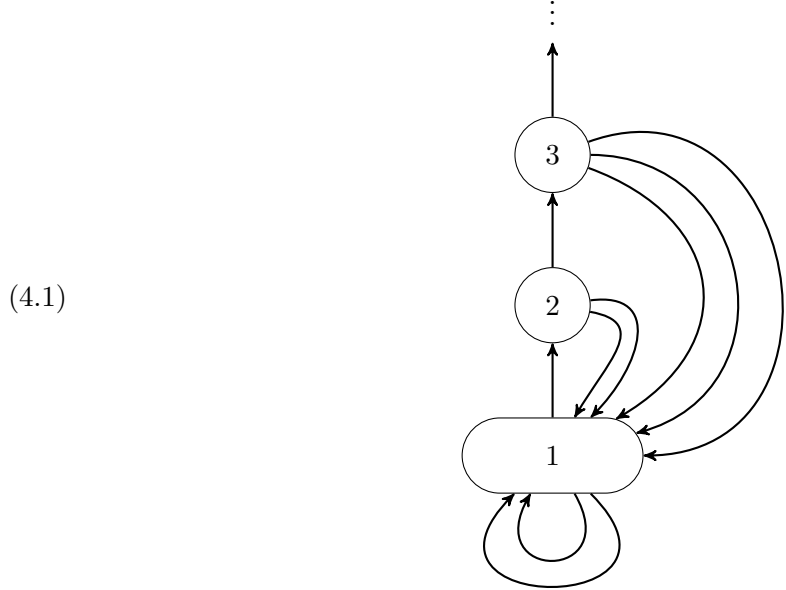
\begin{equation}
    \label{eq:MC}
    \begin{aligned}
        \begin{tikzpicture}[node distance=20mm,minimum size=10mm]
              \node[rounded rectangle,draw] (w1) {\, \qquad 1 \qquad \,};
              \node[circle,draw] (w2) [above of=w1] {2};
              \node[circle,draw] (w3) [above of=w2] {3};
              \node[]            (w4) [above of=w3] {\vdots};
              \path[->,>=stealth',thick]
                (w1) edge (w2)
                (w2) edge (w3)
                (w3) edge (w4)
                (w1) edge [out=-45,in=-135,looseness=5.9] (w1)
                (w1) edge [out=-60,in=-120,looseness=6] (w1)
                (w2) edge [out=7.5,in=45,looseness=1.6] (w1)
                (w2) edge [out=-10,in=60,looseness=1.5] (w1)
                (w3) edge [out=-20,in=30,looseness=1.5] (w1)
                (w3) edge [out=0,in=15,looseness=1.5] (w1)
                (w3) edge [out=20,in=0,looseness=1.8] (w1)
                ;        
        \end{tikzpicture}
    \end{aligned}
\end{equation}
We label the edges from $\circled{n}$ to $\circled{1}$ by elements of some set
$\cR_n$ with $\# \cR_n = r_n$, and the edges $\circled{n} \to \circled{n+1}$ by $E_n$.
Let $\cR$ be the disjoint union $\sqcup_n \cR_n$. Let $\Lambda > 1$.

\begin{rmk}
    In Section~\ref{sec:final} we will associate $\cR_n$, $\cR$ and $\Lambda$ with those from
    Section~\ref{sec:cylinders}, but for now these are abstract.
\end{rmk}

According to~\eqref{eq:MC}, let $\Delta$ be the set of two-sided admissible sequences in the alphabet
$\fA = \cR \cup \{E_n\}_{n \geq 1}$ which visit $\circled{1}$ infinitely often in the future,
and let $\sigma \colon \Delta \to \Delta$ be the left shift.
Naturally, $\sigma$ defines a topological Markov chain.

Let $\Delta_0 \subset \Delta$ be the set of paths which ``start from \circled{1}'',
i.e.\ with the zero-indexed symbol in $\{E_1\} \cup \cR_1$.
Let $\tau \colon \Delta_0 \to \{1,2,\ldots \}$ be the first return time to $\Delta_0$,
\[
    \tau(x) = \min \{ n \geq 1 : \sigma^n(x) \in \Delta_0 \}
    .
\]
Then the induced map $\sigma_\tau \colon \Delta_0 \to \Delta_0$,
$\sigma_\tau(x) = \sigma^{\tau(x)}(x)$ is a full shift on the alphabet
$\cR$. (Technically, the alphabet is 
$\{(E_1, \ldots, E_{|R|-1}, R) : R \in \cR \}$, but it is naturally identified with $\cR$.)

An $x = (\ldots, x_{-1}, x_0, x_1, \ldots) \in \Delta$ can be described by the sequence
of times $t_{n}$ when the orbit of $x$ visits $\Delta_0$, with $t_{-1} < 0 \leq t_0$,
and the corresponding sequence $R_n = x_{t_{n+1}-1} \in \cR_{t_{n+1} - t_n}$ of choices
of return path to $\Delta_0$.
With $x'$ given by $t'_n$ and $R'_n$, define the separation time and separation distance by
\begin{equation}\label{eq:sep}
    \begin{aligned}
        s(x, x') & = \inf \bigl\{ n \geq 0 : t_{\pm n} \neq t'_{\pm n} \text{ or } R_{\pm n} \neq R'_{\pm n} \bigr\}
        , \\
        d(x, x') & = \Lambda^{- s(x, x')}
        .
    \end{aligned}
\end{equation}
Note that the separation time $s(x, x')$ measures time with respect to number of returns to
$\Delta_0$.
Now $d$ is a metric on $\Delta$, and for the rest of this section $\Delta$ is a metric space.

\begin{thm}
    \label{thm:Delta}
    Let $\lambda > 1$ and suppose that $\sum_{n \geq 1} n \lambda^{-n} r_n < \infty$.  
    Then there exists a $\sigma$-invariant probability $\mu_\Delta$ on $\Delta$ with the following properties.
    \begin{enumerate}[label=(\alph*)]

        \item\label{thm:Delta:mix}
            If $r_n = O(\lambda^{n} / n^{\alpha})$ with $\alpha > 2$ and $\gcd \{ n : r_n > 0 \} = 1$,
            then for all $\gamma \in (0,1]$ there exists $C_\gamma > 0$ such that for all $u,v \in C^\gamma(\Delta)$
            we have the correlation bound
            \[
                \biggl| \int u \; v \circ \sigma^n \, d\mu_\Delta - \int u \, d\mu_\Delta \int v \, d\mu_\Delta \biggr|
                \leq C_\gamma |u|_{C^\gamma} |v|_{C^\gamma} \, n^{-\alpha + 2}
                \quad \text{for all} \quad n \geq 1
                .
            \]
        \item\label{thm:Delta:ASIP}
            If $r_n = O(\lambda^{n} / n^{\alpha})$ with $\alpha > 3$ and 
            $v \colon \Delta \to \bR$ is H\"older continuous with $\int v \, d\mu_\Delta = 0$,
            then for each $\eps > (\alpha - 1)^{-1}$
            the partial sums $S_n = \sum_{k=0}^{n-1} v \circ \sigma^k$, as a random process
            on the probability space $(\Delta, \mu_\Delta)$,
            satisfy the ASIP with rate $o(n^{\eps})$.
            
        \item\label{thm:Delta:mu}
            If $\sum_{n \geq 1} \lambda^{-n} r_n = 1$, then
            $\mu_\Delta$ has entropy equal to $\log \lambda$.
    \end{enumerate}
\end{thm}

\begin{rmk}
    There can be at most one nonnegative $\lambda$ satisfying
    the relation $\sum_{n \geq 1} \lambda^{-n} r_n = 1$,
    but its existence is not guaranteed. 
    For example, with $r_n = \lfloor 2^{n-1} / n^{2} \rfloor$
    the sum is smaller than $1$ for $\lambda \geq 2$ and infinite for $\lambda < 2$.
\end{rmk}

\begin{rmk}
    If $r_n = O(\beta^n)$ for some $\beta < \lambda$, then the rate of decay
    of correlations is exponential, following the same proof and using the same
    references. But we will not need this result here.
\end{rmk}

\begin{proof}[Proof of Theorem~\ref{thm:Delta}]
    Set 
    \[
        S = \sum_{n \geq 1} n \lambda^{-n} r_n
        , \qquad
        w_n = S^{-1} \sum_{k \geq n} \lambda^{-k} r_k
        , \qquad
        p_n = \frac{w_n - w_{n+1}}{w_n r_n}
        .
    \]
    First we define $\mu_{\Delta}$ on the space of one-sided paths on~\eqref{eq:MC}
    by prescribing its values on all $n$-cylinders in the alphabet $\fA$.
    For 1-cylinders, let
    \[
        \mu_{\Delta}([E_n]) = w_{n+1}
        \qquad \text{and} \qquad
        \mu_{\Delta}([R]) = S^{-1} \lambda^{-n}
        \; \text{ for } \; R \in \cR_n
        .
    \]
    Then, for a cylinder $[A_1, \ldots, A_k]$ of length at least $2$, let
    \[
        \mu_\Delta([A_1, \ldots, A_k])
        = \mu_\Delta([A_1, \ldots, A_{k-1}]) \Pi(A_{k-1}, A_k)
        ,
    \]
    where the transition probabilities $\Pi(A_{k-1}, A_k)$ are given by
    \[
        \Pi(\cdot, E_n) = \frac{w_{n+1}}{w_n}
        \qquad \text{and} \qquad
        \Pi(\cdot, R) = p_n \; \text{ for } \; R \in \cR_n
        ,
    \]
    as long as the corresponding transitions are compatible with the graph~\eqref{eq:MC},
    and otherwise $\Pi(\cdot, \cdot)$ is zero.

    The above defines $\mu_\Delta$ as a probability measure on the space of one-sided paths.
    Observe that $\mu_\Delta$ is Markov and shift-invariant, hence it extends to
    a $\sigma$-invariant measure on all of $\Delta$.

    To prove~\ref{thm:Delta:mix}, observe that $\sigma \colon \Delta \to \Delta$
    is a suspension over $\sigma_\tau \colon \Delta_0 \to \Delta_0$ with roof function $\tau$.
    The map $\sigma_\tau$ is a Bernoulli shift both topologically
    and measure theoretically (with respect to the $\sigma_\tau$-invariant probability measure
    $\mu_\Delta / \mu_\Delta(\Delta_0)$).
    Hence we deal with a  particularly simple (zero distortion, first return inducing scheme)
    Young tower~\cite{Y98,Y99}. From
    \[
        \gcd \bigl\{ n : \mu_\Delta ( \tau = n) > 0 \bigr\} = 1
    \]
    and the return time tail bound
    \[
        \mu_\Delta \bigl( \tau(x) \geq n \bigr)
         = \frac{w_{n}}{w_1}
         = O(n^{-\alpha + 1})
    \]
    we get the required mixing rate for H\"older observables on $\Delta$,
    see~\cite{KKM19} (based on~\cite{CG12,MT14}).

    The ASIP~\ref{thm:Delta:ASIP} is a standard result for one-sided Young towers \cite{MN05, CDKM20}.
    The ASIP with our particular error rate for general H\"older observables on a two-sided tower is a recent improvement, see~\cite{CDKM23}.

    Finally, we prove \ref{thm:Delta:mu}, assuming $\sum_{n \ge 1} \lambda^{-n} r_n =1$.
    Since $\mu_\Delta$ is a Markov measure, its entropy can be computed explicitly~\cite[Theorem~4.27]{W75}:
    \begingroup
    \allowdisplaybreaks
    \begin{align*}
        h_{\mu_\Delta}
            & = - \sum_{A, B \in \fA} \mu_{\Delta}([A]) \Pi(A, B) \log \Pi(A, B) \\
            & =  - \sum_{n \ge 1} \mu_{\Delta}([E_n])
            \Bigl[ \frac{w_{n+2}}{w_{n+1}} \log \Bigl( \frac{w_{n+2}}{w_{n+1}} \Bigr) + r_{n+1}p_{n+1} \log p_{n+1} \Bigr] \\
            & \qquad \; - \sum_{n\ge 1} \sum_{R \in \cR_n} \mu_{\Delta}([R])
            \Bigl[ \frac{w_2}{w_1} \log \Bigl( \frac{w_2}{w_1} \Bigr) + r_1 p_1 \log p_1  \Bigr] \\
            & = - \sum_{n \ge 1} w_{n+2} \log \Bigl( \frac{w_{n+2}}{w_{n+1}} \Bigr) + (w_{n+1} - w_{n+2}) \log p_{n+1} \\
            & \qquad \; - w_2 \log \Bigl(\frac{w_2}{w_1} \Bigr) + (w_1 - w_2) \log p_1
            ,
    \end{align*}
    \endgroup
    where we have used the fact that $\sum_{n\ge 1} \sum_{R \in \cR_n} \mu_{\Delta}([R])  = w_1$
    and $w_n r_n p_n = w_n - w_{n+1}$.
    Then, using also $p_n = (S \lambda^n w_n)^{-1}$ and $\sum_{n \ge 1} \lambda^{-n} r_n = 1$,
    \begingroup
    \allowdisplaybreaks
    \begin{align*}
        h_{\mu_\Delta}
        & =  - \sum_{n \ge 1} \Bigl[ (w_n - w_{n+1}) \log \frac{\lambda^{-n}}{S w_n}
            + w_{n+1} \log w_{n+1} - w_{n+1} \log w_n
        \Bigr]
        \\
        & =  w_1 \log w_1 - \sum_{n \ge 1} (w_n - w_{n+1})  \log \frac{\lambda^{-n}}{S}
        \\
        & = - \frac{\log S}{S} + \frac{\log S}{S} \sum_{n \ge 1} \lambda^{-n} r_n
        + \frac{\log \lambda}{S} \sum_{n \ge 1} n \lambda^{-n} r_n
        \\
        & =  \log \lambda
        .
    \end{align*}
    \endgroup

 \end{proof}


\section{Proof of Theorems~\ref{thm:decay} and~\ref{thm:ASIP}}
\label{sec:final}

Here we relate the abstract setup of Section~\ref{sec:symbolic} to the billiard
coding from Section~\ref{sec:cylinders}.
As in Section~\ref{sec:cylinders}, let $\cA, \cR$ denote the collections of cylinders,
let $\Lambda$ denote the hyperbolicity constant from \eqref{eq:hyp}, and let $r_n = \# \cR_n$.
In this setting, let $\sigma \colon \Delta \to \Delta$ be the Markov chain
as in Section~\ref{sec:symbolic}.

Recall that $h > s_0 \log 4$ is the topological entropy of the billiard.
In Propositions~\ref{prop:YCb} and \ref{prop:YTb} we established
that $\# \cA_n$ grows as $e^{hn}$, and that
$\# \cR_n = O(e^{hn} / n^{\alpha})$ where $ \alpha = (h - h') / (s_0 \log 2)$,
and $h'$ is sufficiently small so that $\alpha > 2$.
This implies in particular that $\sum_{n \ge 1} n e^{-h n} r_n < \infty$, which is the 
summability condition needed for Theorem~\ref{thm:Delta}.

Since $\# \cA_n > 0$ for all sufficiently large $n$, we have $\gcd \{n : r_n > 0\} = 1$.

We proceed with the more delicate task of verifying the condition needed for
part (c) of Theorem~\ref{thm:Delta}:
there exists $\lambda > 1$ with $\sum_{n \geq 1} \lambda^{-n} r_n = 1$,
and moreover, $\lambda = e^h$.

\begin{prop}
    \label{prop:works}
    $\sum_{n \geq 1} e^{-h n} r_n  = 1$.
\end{prop}

The main ingredients in the proof of Proposition~\ref{prop:works} are the
asymptotics of $\# \cA_n$ and the following simple lemma.

For $x = (x_1, x_2, \ldots)$ let $|x|_1 = \sum_{n \geq 1} |x_n|$
denote the $\ell_1$-norm. We say that $x$ is nonnegative if all
its coordinates are nonnegative.

\begin{lemma}
    \label{lem:Rp}
    Suppose $r'_1, r'_2, \ldots \geq 0$ and
    \[
         R' = \begin{pmatrix}
             r'_1 & 1 \\
             r'_2 & & 1 \\
             r'_3 & & & 1 \\
             \vdots & & & & \ddots
         \end{pmatrix}
         .
    \]
    \begin{enumerate}[label=(\alph*)]
        \item\label{Rp:less}
            If $\sum_{i \ge 1} r_i' < 1$, then $|(R')^n x|_1 \to 0$
            for all nonnegative $x \in \ell_1$.
        \item\label{Rp:moar}
            If $\sum_{i \ge 1} r_i' > 1$, then $|(R')^n x|_1 \to \infty$
            for all nonnegative nonzero $x \in \ell_1$.
    \end{enumerate}
\end{lemma}

\begin{proof}
    Set 
    \[
        \beta = \sum_{i \ge 1} r_i'
        .
    \]
    Without loss of generality, assume that $\beta$ is finite.
    Let $x \in \ell_1$ be nonnegative. Then $|R'x|_1 = |x|_1 + (\beta - 1) x_1$,
    in particular $R' x \in \ell_1$.
    By induction,
    \begin{equation}
        \label{eq:proceed}
        \bigl| (R')^nx \bigr|_1
        = |x|_1 + (\beta - 1) \sum_{k =0}^{n-1} \bigl( (R')^kx \bigr)_1
        .
    \end{equation}
    Using the definition of $R'$ and that $x$ is nonnegative, we have $((R')^k x)_1 \ge x_{k+1}$.

    If $\beta < 1$, then choose $n_1$ sufficiently large that
    $\sum_{k > n_1} x_k \le (\sqrt{\beta} - \beta) |x|_1$.  Then,
    \[
        \bigl| (R')^{n_1}x \bigr|_1
        \le |x|_1 + (\beta-1) \sum_{k =1}^{n_1} x_k \le \sqrt{\beta} |x|_1
        .
    \]
    We iterate this relation, choosing $n_2$ sufficiently large so that 
    $\sum_{k > n_2} (R'x)_k \le (\sqrt{\beta} - \beta) |R' x|_1$.  Then,
    \[
        \bigl| (R')^{n_1+n_2}x \bigr|_1 \le \beta |x|_1
        .
    \]
    Proceeding in this way, we obtain a sequence of integers $n_j$ such that
    \[
        \bigl| (R')^{\sum_{j \le k} n_j} x \bigr|_1
        \le \beta^{k/2} |x|_1
        \quad \text{for each } k \geq 1
        .
    \]
    This shows convergence of $|(R')^n x|_1 \to 0$ along a subsequence.
    Combining this with the fact that when $\beta < 1$, \eqref{eq:proceed} implies
    $|(R')^n x |_1 \le |x|_1$
    whenever $x$ is nonnegative, we obtain~\ref{Rp:less}.

    To prove~\ref{Rp:moar} when $\beta >1$, we proceed similarly,
    choosing $n_1$ sufficiently large so that 
    $\sum_{k > n_1} x_k \le \frac{\beta - \sqrt{\beta}}{\beta -1 } |x|_1$.
    Then applying this to~\eqref{eq:proceed},
    \[
        \bigl| (R')^{n_1}x \bigr|_1
        \ge |x|_1 + (\beta-1) \sum_{k=1}^{n_1} x_k \ge \sqrt{\beta} |x|_1
        .
    \]
    Iterating as before we show that $|(R')^n x|_1 \to \infty$ along a subsequence.
    Using that when $\beta > 1$, \eqref{eq:proceed} implies $|(R')^n x |_1 \ge |x|_1$
    whenever $x$ is nonnegative, we obtain~\ref{Rp:moar}.
\end{proof}

\begin{proof}[Proof of Proposition~\ref{prop:works}]
    Define the connectivity matrix for our Markov chain
    \[
         R = \begin{pmatrix}
             r_1 & 1 \\
             r_2 & & 1 \\
             r_3 & & & 1 \\
             \vdots & & & & \ddots
         \end{pmatrix}
         .
    \]
    For a vector $x = (x_1, x_2, \ldots)$, define 
    $x' = (x'_1, x'_2, \ldots) = (x_1 e^{-h}, x_2 e^{-2h}, \ldots)$.
    In other words, $x' = H x$ where $H = \diag(e^{-h}, e^{-2h}, \ldots)$.
    Accordingly, let $r = (r_1, r_2, \ldots)$ and $r' = H r$.

    To the transformation $x \mapsto e^{-h} R x$ there corresponds the
    conjugate transformation
    \[
        \begin{pmatrix} x'_1 \\ x'_2 \\ \vdots \end{pmatrix}
        \mapsto R' x' = x'_1 \begin{pmatrix} r'_1 \\ r'_2 \\ \vdots \end{pmatrix}
        + \begin{pmatrix} x'_2 \\ x'_3 \\ \vdots \end{pmatrix}
        \text{, where }  R' = \begin{pmatrix}
            r_1' & 1 \\
            r_2' & & 1 \\
            r_3' & & & 1 \\
            \vdots & & & & \ddots
        \end{pmatrix}
        .
    \]
    That is, $e^{-h} H R x = R' H x$.
    Since $H$ is invertible, in fact diagonal, we have 
    $(R')^n = e^{-hn}H R^n H^{-1}$ for each $n$.
    In particular, for each $i \ge 1$,
    \begin{equation}
        \label{eq:R' R}
        \bigl( (R')^n \bigr)_{i,1}
        = e^{-hn} e^{-h(i-1)} (R^n)_{i,1}
        .
    \end{equation}
    
    Let $C$ denote various positive constants which depend only on $\{r_n\}$ and $\alpha$.

    Let $x = (1,0,0,\ldots)$.
    We claim that for all sufficiently large $n$,
    \begin{equation}
        \label{eq:RRR'}
        C^{-1}
        \leq \bigl| (R')^n x \bigr|_1
        \leq C
        .
    \end{equation}
    Then by Lemma~\ref{lem:Rp}, $\sum_{i \ge1} r_i' = \sum_{i \ge 1} e^{-hi} r_i = 1$, as required.

    It remains to verify~\eqref{eq:RRR'}. The lower bound is straightforward: observing that
    $(R^n)_{1,1} = \# \cA_n$ and using the bound $\# \cA_n \geq C e^{hn}$ from Proposition~\ref{prop:YCb},
    for all sufficiently large $n$,
    \[
        \bigl| (R')^n x \bigr|_1
        \geq \bigl( (R')^n \bigr)_{1,1}
        = e^{-h n} (R^n)_{1,1}
        = e^{-h n} \# \cA_n
        \geq C
        .
    \]
    To justify the upper bound, observe that
    $(R^n)_{i,1}$ denotes the number of paths of length $n$ from $\circled{i}$ to $\circled{1}$.
    Thus, with a convention $\# \cA_0 = 1$,
    \[
        (R^n)_{i,1} = \sum_{t=1}^n r_{i+t-1} \# \cA_{n-t}
        .
    \]
    Further, using
    the bounds $r_n \leq C e^{h n} / n^\alpha$ from Proposition~\ref{prop:YTb}
    and $\cA_n \leq C e^{hn}$ from Proposition~\ref{prop:YCb},
    \begin{equation}
        \label{eq:R entry}
        (R^n)_{i,1}
        \leq C \sum_{t=1}^n e^{ h(n + i)} / (i+t-1)^\alpha
        .
    \end{equation}
    By~\eqref{eq:R' R},
    \begin{equation}
        \label{eq:R'1}
        \bigl| (R')^n x \bigr|_1
        = \sum_{i \ge 1} \bigl( (R')^n \bigr)_{i,1}
        = \sum_{i \ge 1} e^{-hn} e^{-h(i-1)} (R^n)_{i,1}
        .
    \end{equation}
    Assembling~\eqref{eq:R' R} and~\eqref{eq:R entry} and using $\alpha >2$,
    \begin{align*}
        \bigl|( R' \bigr)^n x|_1
        & \le C \sum_{i \ge 1} \sum_{t=1}^n (i+t-1)^{-\alpha}
        \le C \sum_{i \ge 1} i^{-\alpha+1}
        \le C
        .
    \end{align*}
    The upper bound in~\eqref{eq:RRR'} is proved.
\end{proof}

At this point, the assumptions of Theorem~\ref{thm:Delta} are verified, 
namely $\sum_n e^{-hn} r_n = 1$ for part~\ref{thm:Delta:mu}, and $r_n = O(e^{hn} / n^\alpha)$,
assuming that
$\alpha = \frac{h - h'}{ s_0 \log 2} > 2$ for part~\ref{thm:Delta:mix} and $\alpha > 3$ for part~\ref{thm:Delta:ASIP}.
Both of these imply that $\sum_n n e^{-h n} r_n < \infty$.

It remains to associate $\mu_\Delta$ with the measure of maximal entropy $\mu_0$ on $M$,
and to show that decay of correlations and ASIP on $\Delta$ translate
to the billiard.

There is a natural association between paths on $\Delta$ and orbits of $T$.
We make this formal by defining the semiconjugacy $\pi \colon \Delta \to M$
(i.e.\ $\pi \circ \sigma = T \circ \pi$). First we define $\pi$ on $\Delta_0$.
Suppose that $x \in \Delta_0$ follows forward trajectory $R_0, R_1, \ldots$
in the sense that
\[
    x = (\ldots, x_{-2}, x_{-1}, E_1, \ldots, E_{|R_0|-1}, R_0, E_1, \ldots, E_{|R_1|-1}, R_1, \ldots)
    .
\]
On the billiard, let $Q_n$ be the maximal s-subrectangle of $D$ contained
in $D \cap R_0 \cdots R_n$.
Here $D$ is the rectangle from~(\ref{A:D}) which we use as a ``base'' for symbolic dynamics,
and $R_0 \cdots R_n$ denotes concatenation of cylinders.
Recall that each $R_k$ contains an s-subrectangle of $D$
whose image under $T^{|R_k|}$ u-crosses $D$.

According to our construction, $Q_n$ are nonempty, closed and nested;
different sequences $R_0, \ldots, R_n$ result in disjoint $Q_n$.
Due to the hyperbolicity of the billiard map, the $Q_n$ are exponentially thin
in the unstable direction (while fully crossing $D$ in the stable direction), so
$\lim_{n \to \infty} Q_n$ is a stable leaf fully crossing $D$.

We see that for $x \in \Delta_0$, its future symbolic itinerary defines a stable leaf fully
crossing $D$, say $W^s$. (And to different $x$ there correspond different leaves.)
Similarly, the past itinerary of $x$ defines a unique unstable leaf $W^u$
fully crossing $D$. Set $\pi(x) = W^s \cap W^u$.
Thus we define $\pi$ on $\Delta_0$, and the semiconjugacy property defines an extension to the whole of $\Delta$:
for a general $x \in \Delta$ there is $n$ such that $\sigma^n (x) \in \Delta_0$,
and we set $\pi(x) = T^{-n} \pi(\sigma^n(x))$.

Observe that by construction:
\begin{itemize}
    \item $\pi$ is indeed a semiconjugacy, i.e.\ $\pi \circ \sigma = T \circ \pi$,
    \item $\pi$ in injective on $\Delta_0$, and
    \item $\pi(\Delta_0) \subset D_*$, where $D_*$ is the Cantor rectangle corresponding to $D$
        as in Definition~\ref{def:cantor}.
\end{itemize}

\begin{lemma}
    \label{lem:buz99}
    If $\nu$ is a probability measure on $\Delta$
    and $\mu = \pi_* \nu$, then $\nu$ and $\mu$ have equal entropies.
    In particular, $\pi_* \mu_\Delta$ is the measure of maximal entropy $\mu_0$.
\end{lemma}

\begin{proof}
    Restricted to $\Delta_0$, the map $\pi$ is injective.
    Moreover, since $T$ and $\sigma$ are invertible, $\pi$ is injective on
    $\sigma^n(\Delta_0)$ for each $n$.
    Since $\Delta = \cup_n \sigma^n(\Delta_0)$ and
    every $z \in M$ has at most one preimage in each $\sigma^n(\Delta_0)$,
    we observe that $\pi$ is at most countable-to-one.
    By~\cite[Proposition~2.8]{B99}, the entropy of $\pi_* \mu_\Delta$ is equal to
    that of $\mu_\Delta$, and $\pi_* \mu_\Delta = \mu_0$ by
    uniqueness of the measure of maximal entropy~\cite[Theorem~2.4]{BD20}.
\end{proof}

\begin{lemma}
    \label{lem:hyp}
    $\pi$ is Lipschitz continuous.
\end{lemma}

\begin{proof}
    Let $x, x' \in \Delta$ with the associated times $t_n, t'_n$ of visits to $\Delta_0$
    and separation time $s(x,x')$ as in~\eqref{eq:sep}.
    Using that $|t_j| \ge j$ and hyperbolicity~\eqref{eq:hyp}, one has
    \[
        d(\pi(x) , \pi(x')) \le C_e^{-1} \Lambda^{- \min \{ t_{s(x,x')-1}, - t_{-s(x,x')+1} \} }
        \le \Lambda C_e^{-1} d(x,x')
        .
    \]
\end{proof}

Lemmas~\ref{lem:buz99} and~\ref{lem:hyp} complete the proof of
Theorems~\ref{thm:decay} and~\ref{thm:ASIP}.
Indeed, if $u, v \colon M \to \bR$ are H\"older observables on $M$,
then their lifts $\tilde{u} = u \circ \pi$ and $\tilde{v} = v \circ \pi$ are H\"older on $\Delta$.
From $\pi_* \mu_\Delta = \mu_0$ and $v \circ T^n \circ \pi = \tilde{v} \circ \sigma^n$, we have
\[
    \int_M u \, v \circ T^n \, d\mu_0
    = \int_{\Delta} \tilde{u} \, \tilde{v} \circ \sigma^n \, d\mu_{\Delta}
    .
\]
Similarly, the random process $(v \circ T^n)_n$ on the probability space $(M, \mu_0)$
is equal in law to the random process $(\tilde{v} \circ \sigma^n)_n$ on $(\Delta, \mu_\Delta)$.
Thus decay of correlations and the ASIP for H\"older observables on $M$ follow from the
corresponding results on $\Delta$.

Theorems~\ref{thm:decay} and~\ref{thm:ASIP} are proved.


\subsection{Super-polynomial mixing for typical dispersing billiard tables}
\label{sec:superpoly}

In this section, we prove Corollary~\ref{cor:super}. That is, we prove that
the rate of mixing for a dispersing billiard is super-polynomial
if the sequence of complexities is bounded:
\begin{equation}
\label{eq:conjecture}
\mbox{There exists $K>0$ such that $K_n \le K$ for all $n \ge 0$,}
\end{equation}
where $K_n$ denotes the maximal number of curves in $\cS_n$ intersecting at one point. 
Assuming Conjecture~\ref{conj}, this property holds for typical configurations
of finite horizon dispersing billiards.

The main consequence of~\eqref{eq:conjecture} is that the parameter $s_0$ in~\eqref{eq:sparse}
can be chosen as small as desired by choosing $n_0$ large and $\vf_0$ close to $\pi/2$.

\begin{prop}
    \label{prop:s_0}
    If the billiard table satisfies \eqref{eq:conjecture},
    then for any $\ve_0 > 0$, there exists $n_0>0$ and $\vf_0 \in (0, \pi/2)$
    such that $s_0(\vf_0, n_0) < \ve_0$.
\end{prop}

\begin{proof}
    Let $\ve_0 > 0$ and choose $n_0 \in \mathbb{N}$ such that $K/n_0 < \ve_0$.
    The singularity set $\cS_{n_0}$ contains finitely many curves,
    and thus finitely many intersection points of these curves, 
    which we label $\{ z_i \}_{i=1}^L$, where $L$ depends on $n_0$ 
    (indeed, $L$ grows exponentially as a function of $n_0$).
    We denote by
    $N_\ve(\cdot)$ the $\ve$-neighborhood of a set in $M$. 

    By assumption, the number of curves in $\cS_{n_0}$ intersecting at each $z_i$ is at most $K$.  So 
    for $\ve > 0$ sufficiently small, the 
    $\ve$-neighborhood of $z_i$ is split into at most $K+1$ sectors by these curves.  
    In particular, the curves intersecting at $z_i$ belong to $\cup_j T^{-\tau_{ij}}(\cS_0)$ for
    at most $K$ times $\tau_{ij} \in [0, n_0]$.

    By the continuity of $T^{n_0}$ on each sector, there exist $\ve_1(z_i), \ve_2(z_i)>0$ such that
    if $x \in N_{\ve_1}(z_i)$, then $T^n(x)$ can only enter $N_{\ve_2}(\cS_0)$ at the times $\tau_{ij}$.
    For $j=1,2$, set $\ve_j = \min_i \ve_j(z_i)$, and note that $\ve_1, \ve_2 > 0$.
    Then
    \[
        \# \bigl\{0 \leq n \leq n_0 : T^n(x) \in N_{\ve_2}(\cS_0) \bigr\} \le K 
        \qquad \text{for all} \qquad
        x \in \cup_i N_{\ve_1}(z_i)
        .
    \]

    Now we turn our attention to the elements of $\cM_{n_0}$,
    the partition of $M \setminus \cS_{n_0}$ into connected components.
    Let $A \in \cM_{n_0}$ and $A^\diamond = A \setminus \cup_i N_{\ve_1}(z_i)$.
    For each $x \in \partial A^\diamond \cap \cS_{n_0}$, we have
    $\# \{ 0 \leq n \leq n_0 : T^n(x) \in \cS_0 \} \leq 1$.
    Remark that each connected component  $S \in \partial A^\diamond \cap \cS_{n_0}$ is adjacent to 
    $A$ and another cell, $A'$, and $T^{n_0}$ is continuous on either $A \cup S$ or $A' \cup S$.  In the former
    case, there exist $\ve_3, \ve_4 > 0$ such that 
    \[
        \# \bigl\{0 \leq n \leq n_0 : T^n(x) \in N_{\ve_4}(\cS_0) \bigr\} = 1
        \qquad \text{for all} \qquad
        x \in N_{\ve_3} (S)
        .
    \]
    In the latter case, there exist $\ve_3, \ve_4>0$ such that
    \[
        \# \bigl\{0 \leq n \leq n_0 : T^n(x) \in N_{\ve_4}(\cS_0) \bigr\} = 0
        \qquad \text{for all} \qquad
        x \in N_{\ve_3} (S)
        .
    \]
    Finally, by continuity of $T^{n_0}$ on $A \setminus N_{\ve_3}(\partial A)$, there exists $\ve_5>0$ such that
    \[
        \# \bigl\{0 \leq n \leq n_0 : T^n(x) \in N_{\ve_5}(\cS_0) \bigr\} = 0
        \qquad \text{for all} \qquad
        x \in A \setminus N_{\ve_3} (\partial A)
        .
    \]
    By the finiteness of $\cM_{n_0}$ and $\cS_{n_0}$, we may choose $\ve_3, \ve_4, \ve_5 > 0$
    which work for every $A \in \cM_{n_0}$ and $S$ as above.

    Set $\vf_0 = \pi/2 - \min \{ \ve_1, \ve_2, \ve_3, \ve_4, \ve_5 \}$.
    Then $s_0(\vf_0, n_0) \le K/n_0 < \ve_0$ as required. 
\end{proof}

From the form of the exponent in Theorem~\ref{thm:decay}, Proposition~\ref{prop:s_0}
immediately implies that the rate of decay of correlations for $\mu_0$ is super-polynomial and
Corollary~\ref{cor:super} is proved.

\section*{Acknowledgements}
The authors are deeply grateful to:
\begin{itemize}
    \item V.~Baladi for suggesting and initiating this project,
    \item P.~B\'alint and I.P.~T\'oth for discussions regarding
        the complexity conjecture for typical dispersing billiards,
    \item J.~Carrand for making precise the use of $s_0$ in~\eqref{A:super-growth},
    \item N.~Dobbs, J.~Sedro, F.~S\'elley and C.~Wormell for helpful discussions,
    \item Fairfield University for its hospitality during March 2022,
    \item The anonymous referee for many helpful suggestions and clarifications.
\end{itemize}
MD is partially supported by NSF grant DMS 2055070.  AK was supported by
the European Research Council (ERC) under the European Union's Horizon 2020
research and innovation programme (grant agreement No 787304).
While working on corrections, AK was supported by EPSRC grant EP/V053493/1.

\appendix

\section{Induced polynomial tails}

In this appendix we prove the following proposition:

\begin{prop}
    \label{prop:ktails}
    Suppose that $X_1, X_2, \ldots$ and $\tau$ are $\{0, 1, 2, \ldots \}$-valued random variables
    with 
    \[
        \bP (X_{n+1} = k \mid \tau \geq n + 1, X_1, \ldots, X_n) \leq A k^{-\alpha}
        \quad \text{for all} \quad
        k \geq 1 , \ n \geq 0
        ,
    \]
    and
    \[
        \bP (\tau \geq n + 1 \mid \tau \geq n, X_1, \ldots, X_n) \leq \theta
        \quad \text{for all} \quad
        n \geq 1
        ,
    \]
    where $\alpha > 1$, $A > 0$ and $\theta \in (0,1)$.
    Let $S = \sum_{n \leq \tau} X_n$.
    Then there is $C > 0$, depending only on $A$, $\alpha$ and $\theta$,
    such that $\bP(S = k) \leq C k^{-\alpha}$ for all $k \geq 1$.
\end{prop}

\begin{proof}
    Let $\eta = \theta^{1/2}$.
    Fix $B, K, M, N$ large so that
    \[
        A \theta \eta^{-1} C_M
        \leq \frac{1 - \theta^{1/4}}{4}
        , \quad \text{where} \quad
        C_M = 2^{1 + \alpha} \sum_{\ell \geq M} \ell^{-\alpha}
        ,
    \]
    \[
        \frac{N + M}{N} \frac{(K - M)^{-\alpha}}{K^{-\alpha}}
        \leq \theta^{-1/4}
        , \qquad
        \frac{A}{B} \theta^{-1/4}
        \leq \min \Bigl\{ \theta^{1/4}, \frac{1-\theta^{1/4}}{4} \Bigr\} 
        .
    \]
    Denote $S_n = \sum_{k=1}^n X_k$.
    Observe that the following bound holds for $n = 1$:
    \begin{equation}
        \label{eq:PSn}
        \bP(a \leq S_n < b, \tau \geq n)
        \leq B \eta^n (b-a) a^{-\alpha}
        \quad \text{for all} \quad a \geq K \quad \text{and} \quad b-a \geq N
        .
    \end{equation}
    Suppose that~\eqref{eq:PSn} holds for $n \geq 1$.
    Estimate for $a \geq K$ and $b - a \geq N$:
    \begin{align*}
        \bP( a - M \leq S_n < b, \tau \geq n + 1)
        & = \bP(\tau \geq n + 1 \mid a - M \leq S_n < b, \tau \geq n)
        \bP(a - M \leq S_n < b, \tau \geq n)
        \\
        & \leq \theta B \eta^n (b-a+M) (a-M)^{-\alpha}
        .
    \end{align*}
    \begin{align*}
        \bP(a - M \leq X_{n+1} < b, \tau \geq n + 1)
        & = \bP(a - M \leq X_{n+1} < b \mid \tau \geq n + 1) \bP(\tau \geq n + 1)
        \\
        & \leq A (b-a+M) (a-M)^{-\alpha} \theta^{n+1}
        .
    \end{align*}
    \begin{align*}
        \sum_{M \leq \ell < a - M}
        & \bP(X_{n+1} = \ell, a - \ell \leq S_n < b - \ell, \tau \geq n + 1)
        \\
        & = \sum_{M \leq \ell < a - M}
        \bP(X_{n + 1} = \ell \mid a - \ell \leq S_n < b - \ell, \tau \geq n + 1)
        \\
        & \qquad \bP(\tau \geq n + 1 \mid a - \ell \leq S_n < b - \ell, \tau \geq n)
        \bP(a - \ell \leq S_n < b - \ell, \tau \geq n)
        \\
        & \leq A B \theta \eta^n (b-a) \sum_{M \leq \ell < a - M}
        \ell^{-\alpha} (a - \ell)^{-\alpha}
        \leq A B \theta \eta^n (b-a) a^{-\alpha} C_M         
        .
    \end{align*}
    Using the three estimates above,
    \begin{align*}
        \bP( a \leq S_{n+1}
        & < b, \tau \geq n + 1)
        \\
        & \leq \bP( a - M \leq S_n < b, \tau \geq n + 1)
        + \bP(a - M \leq X_{n+1} < b, \tau \geq n + 1)
        \\
        & \qquad + \sum_{M \leq \ell < a - M}
        \bP(X_{n+1} = \ell, a - \ell \leq S_n < b - \ell, \tau \geq n + 1)
        \\
        & \leq B \eta^{n+1} (b-a) a^{-\alpha}
        \Bigl(
        \theta \eta^{-1} \frac{b - a + M}{b - a} \frac{(a - M)^{-\alpha}}{a^{-\alpha}}
        \\
        & \qquad + \theta^{n+1} \eta^{-(n+1)}
        \frac{A}{B} \frac{b - a + M}{b - a} \frac{(a - M)^{-\alpha}}{a^{-\alpha}}
        + A \theta \eta^{-1} C_M
        \Bigr)
        \\
        & \leq B \eta^{n+1} (b-a) a^{-\alpha}
        .
    \end{align*}
    This creates an induction step which shows that~\eqref{eq:PSn} holds for all $n$.    

    Now, for $k \geq K$,
    \begin{align*}
        \bP(S = k)
        & = \sum_{n \geq 1} \bP(S_n = k, \tau = n)
        \leq \sum_{n \geq 1} \bP(k \leq S_n < k + N, \tau \geq n)
        \\
        & \leq \sum_{n \geq 1} B \eta^n N k^{-\alpha}
        \leq \frac{B N \eta}{1 - \eta} k^{-\alpha}
        .
    \end{align*}
    The result follows.
\end{proof}

\end{document}